\documentclass[onefignum,onetabnum]{siamart190516}

\usepackage{braket}

\usepackage{array}

\Crefname{ALC@unique}{Line}{Lines}

\usepackage{hyperref}
\usepackage{url}

\usepackage[utf8]{inputenc} 
\usepackage[T1]{fontenc}    
\usepackage{amsfonts}       
\usepackage{nicefrac}       
\usepackage{microtype}      

\usepackage{algorithm}
\usepackage{algpseudocode}

\usepackage{multirow}
\usepackage{graphicx}
\usepackage[caption=false]{subfig}
\usepackage{amsmath,amssymb}
\usepackage{xcolor}
\usepackage[numbers,square,comma,sort&compress]{natbib}

\usepackage{booktabs} 
\usepackage{soul}  
\usepackage[shortlabels]{enumitem}

\newcommand{\cO}{\mathcal{O}}

\newcommand{\bfy}{{\bf y}}
\newcommand{\bfz}{{\bf z}}
\newcommand{\bfv}{{\bf v}}
\newcommand{\bfg}{{\bf g}}
\newcommand{\bfe}{{\bf e}}
\newcommand{\supp}{\mathrm{supp}}

\DeclareMathOperator*{\argmin}{\mathrm{argmin}}
\DeclareMathOperator*{\minimize}{\mathrm{minimize}}

\newsiamthm{conjecture}{Conjecture}

\newsiamthm{example}{Example}
\newsiamthm{problem}{Problem}
\newsiamthm{question}{Question}
\newsiamthm{assumption}{Assumption}
\newsiamremark{remark}{Remark}
\newsiamremark{hypothesis}{Hypothesis}

\newcommand{\proj}{\mathrm{P}}
\newcommand\numberthis{\addtocounter{equation}{1}\tag{\theequation}}
\newcommand{\tnabla}{\tilde{\nabla}}

\title{Zeroth-Order Regularized Optimization (ZORO): 
Approximately Sparse Gradients and Adaptive Sampling\thanks{This paper has been accepted to SIAM Journal on Optimization (SIOPT) and will be published electronically soon.
}
}

\author{HanQin Cai\thanks{Department of Mathematics, University of California, Los Angeles, Los Angeles, CA, USA (\email{hqcai@math.ucla.edu}).}
\and Daniel Mckenzie\thanks{Department of Mathematics, University of California, Los Angeles, Los Angeles, CA, USA (\email{mckenzie@math.ucla.edu}).} \and Wotao Yin\thanks{Damo Academy, Alibaba US,
Bellevue, WA, USA (\email{wotao.yin@alibaba-inc.com}).}
\and Zhenliang Zhang\thanks{Xmotors AI, Mountain View, CA, USA (\email{zhenliang.zhang@gmail.com}).}}

\headers{Zeroth-Order Regularized Optimization (ZORO)}{HanQin Cai, Daniel Mckenzie, Wotao Yin, and Zhenliang Zhang}

\begin{document}

\maketitle

\begin{abstract}
    We consider the problem of minimizing a high-dimensional objective function, which may include a regularization term, using only noisy evaluations of the function. Such optimization is also called derivative-free, zeroth-order,  or black-box optimization. We propose a new \textbf{Z}eroth-\textbf{O}rder \textbf{R}egularized \textbf{O}ptimization method, dubbed ZORO. 
    When the underlying gradient is approximately sparse at an iterate, ZORO needs very few objective function evaluations to obtain a new iterate that decreases the objective function. We achieve this with an adaptive, randomized gradient estimator, followed by an inexact proximal-gradient scheme.  
    Under a novel approximately sparse gradient assumption and various different convex settings, we show the (theoretical and empirical) convergence rate of ZORO is only logarithmically dependent on the problem dimension. Numerical experiments show ZORO outperforms existing methods on both synthetic and real datasets.   
\end{abstract}

\begin{keywords}
zeroth-order optimization, black-box optimization, derivative-free optimization, compressible gradients, sparse gradients, sparse adversarial attack
\end{keywords}
\begin{AMS}
90C56, 	65K05 , 68T05, 68Q25
\end{AMS}

\section{Introduction}
\label{sec:Introduction}
Zeroth-order optimization, also known as derivative-free or black-box optimization, appears in a wide range of applications where either the objective function is implicit or its gradient is impossible or too expensive to compute. These applications include structured prediction \citep{taskar2005learning}, 
reinforcement learning \citep{choromanski2018structured}, 
bandit optimization \citep{flaxman2004online,shamir2013complexity}
optimal setting search in material science experiments \citep{Nakamura2017}, 
adversarial  attacks on neural networks \citep{kurakin2016adversarial,papernot2017practical}, 
and hyper-parameter tuning \citep{snoek2012practical}. In this work, we propose a new method, which we coin ZORO, for high dimensional {\em regularized} zeroth-order optimization problems:
\begin{equation} \label{eq:reg_opt_problem}
    \minimize_{x\in\mathbb{R}^d} F(x):=f(x)+r(x),
\end{equation}
where $r$ is an \emph{explicit} convex extended real-valued function ({\em i.e.} $r: \mathbb{R}^d\to \mathbb{R}\cup\{+\infty\}$) and $f$ is accessible only via a noisy zeroth-order oracle:
\begin{equation}
E_f(x)=f(x) +\xi,  
\label{eq:OracleDefinition}
\end{equation}
where $\xi$ is the unknown oracle noise. When we call the oracle with an input $x$, it returns $E_f(x)$ in which $\xi$ changes every time. Employing a regularizer allows us to use prior knowledge about the problem structure explicitly, without expending additional queries. For example, regularizers can be used to enforce or encourage non-negativity ($x\ge 0$), box constraints ($\ell \le x \le u$), or solution sparsity ($\|x\|_1\leq s$). Allowing $r$ to be extended real-valued means any constrained problem: $\minimize_{x\in\mathcal{X}} f(x)$ with $\mathcal{X}$ convex can be reduced to \eqref{eq:reg_opt_problem}; just take $r$ to be the indicator function $\delta_{\mathcal{X}}$ defined as:
\begin{equation*}
    \delta_{\mathcal{X}}(x) = \left\{\begin{array}{lc} 0, &  x \in \mathcal{X} \\ +\infty, &  x \notin \mathcal{X}\end{array} \right..
\end{equation*}

\paragraph{Gradient compressibility} Since queries are typically assumed to be expensive, the appropriate metric for comparing zeroth-order methods is the number of oracle queries needed to achieve a target accuracy. In order to find an $\varepsilon$-optimal solution, a generic zeroth-order algorithm requires at least $\Omega(d/\varepsilon^{2})$ queries \citep{jamieson2012query}. When $d$ is large, this cost can be prohibitive. ZORO reduces the dependence on $d$ from linear to logarithmic by exploiting {\em gradient compressibility}, by which we mean that the sorted components of $\nabla f(x)$ decay like $1/i^q$ for some exponent $q$. Gradient compressibility is largely unexplored in the zeroth-order optimization community, although we note the works \citep{Wang2018,Balasubramanian2018} which exploit the more restrictive {\em gradient sparsity} assumption (see Assumption~\ref{Assumption:ExactSparsity}). However, in many common applications of zeroth-order optimization (for example, hyperparameter tuning \citep{bergstra2012random} and simulation-based optimization \citep{knight2007association}), it has been empirically observed that $f(x)$ is sensitive to only a few variables at a time. These variables thus carry significantly larger weights in $\nabla f(x)$, making it compressible. Our own experiments (see Figure~\ref{fig:compressible_gradient}) reinforce the idea that gradient compressibility is surprisingly ubiquitous in real-world problems. We emphasize that by assuming gradient compressibility instead of gradient sparsity, we are allowing for completely dense gradients (see Figure~\ref{fig:compressible_gradient}). Moreover, the subset of indices corresponding to the largest entries in $\nabla f(x)$ can change for different $x$. \\

\begin{figure}[t]
\centering
\hfill
\subfloat[Asset management]{
\includegraphics[scale = 0.4]{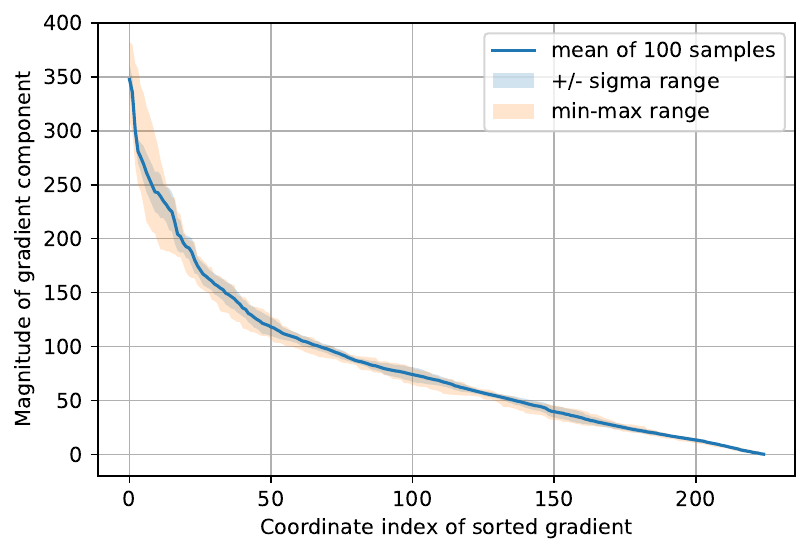}} \hfill
\subfloat[Imagenet adversarial attack]{
\includegraphics[scale = 0.4]{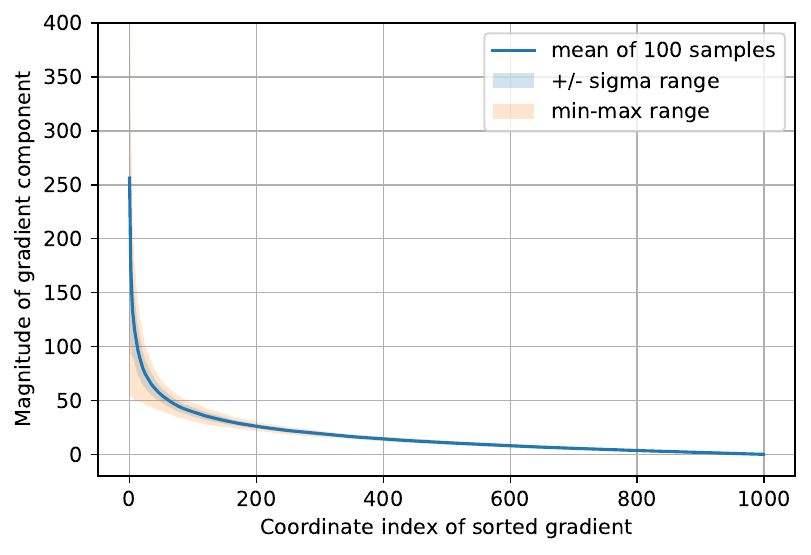}
\label{fig:Imagenet Grad}}
\hfill
\vspace{-0.05in}
\caption{Sorted gradient components at $100$ random points in real-world optimization problems. In Figure~\ref{fig:Imagenet Grad} the gradient is computed within a randomly selected $1000$-dimensional subspace. Such decays indicate the gradients are compressible. } \label{fig:compressible_gradient}
\end{figure}

\begin{figure}[t]
\centering
\hfill
\subfloat[Asset management]{
\includegraphics[scale = 0.4]{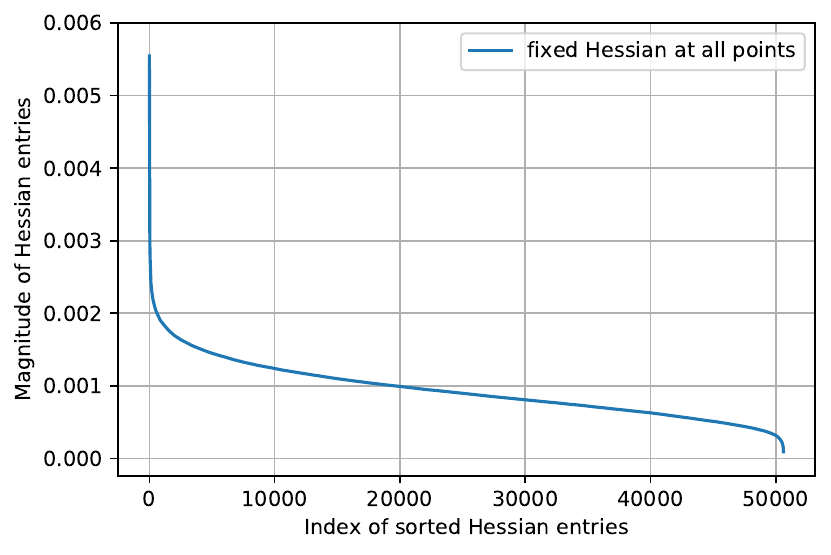}} 
\hfill
\subfloat[Imagenet adversarial attack]{
\includegraphics[scale = 0.4]{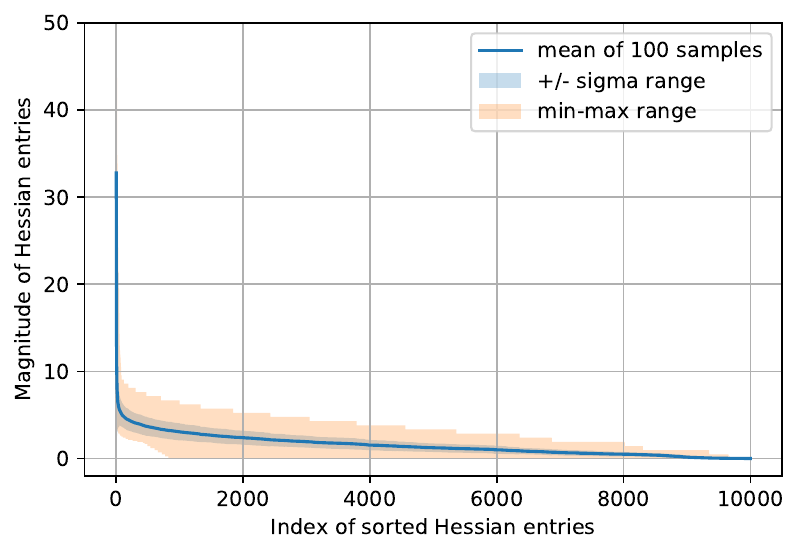} \label{fig:Imagenet Hessian}}
\hfill
\vspace{-0.05in}
\caption{Sorted Hessian entries at $100$ random points in real-world optimization problems. Such decays indicate the Hessians are weakly sparse. In Figure~\ref{fig:Imagenet Hessian} the gradient is computed within a randomly selected $1000$-dimensional subspace. Note that the asset management problem has fixed Hessian at all points, thus there is no variance. } \label{fig:compressible_hessian}
\end{figure}

\paragraph{Inexact prox-gradient descent} ZORO's main iteration is based on the prox-gradient descent method but uses an approximate gradient $\hat{\bfg}_k \approx \nabla f(x_k)$ constructed using randomized finite differences and compressed sensing. Key to our analysis is a careful estimation of the gradient error, $\|\nabla f(x_k) - \hat{\bfg}_k\|_2$, which comes from the following four components: the oracle noise, the finite differencing, the error due to compressed sensing, and the ``tail error'' due to approximating $\nabla f(x_k)$ by its best $s$-sparse approximation. Although there are many prior results in the literature characterizing the iteration complexity of (prox-) gradient descent using an inexact gradient \citep{schmidt2011convergence,friedlander2012hybrid,tappenden2016inexact,berahas2021global}, they do not precisely fit our situation. For example, \citep{schmidt2011convergence} requires $\|\nabla f(x_k) - \hat{\bfg}_k\|_2\leq \varepsilon_k$ with $\sum_{k=1}^{\infty}\varepsilon_k < \infty$, which is not the case for ZORO. Thus we prove new results on convergence of inexact (prox-) gradient descent that may be of independent interest. \\

\paragraph{Lower query complexity} Our improved analysis of prox-gradient descent, together with our gradient error estimates, enable us to prove ZORO exceeds the state-of-the-art in terms of query complexity, at least for functions exhibiting gradient sparsity or compressibility. Specifically, for convex $f$, ZORO finds an $\varepsilon$-optimal solution using only $\cO(s\log(d)/\varepsilon)$ queries. For restricted strongly convex $f$, this complexity improves to $\cO(s\log(d)\log(\varepsilon^{-1}))$. For non-convex $f$, ZORO finds an $\varepsilon$-stationary point using $\cO(s\log(d)/\varepsilon^2)$ queries. We mention two caveats to these results. Firstly, they require $\varepsilon \geq \varepsilon_0$ where $\varepsilon_0$ is a constant controlled primarily by the magnitude of the oracle noise $\xi$. As we do not assume zero-mean oracle noise, such a lower bound on achievable accuracy is unavoidable. Secondly, as ZORO incorporates some stochasticity, these results are probabilistic. However, the probability of failure is so infinitesimally small, it is almost certain to never occur.\\

\paragraph{Adaptive sampling} Empirically, we have observed gradients are less compressible for $x$ closer to the optimal solution. Hence, practical methods need to dynamically choose the number $s$ of large components of $\nabla f(x)$ to target. To this end, we introduce AdaZORO, a version of ZORO employing an {\em adaptive sampling} strategy. In the absence of gradient compressibility AdaZORO reduces to a fixed step-size descent method using the ``linear interpolation'' gradient estimator shown to be effective in \citep{berahas2021theoretical}. So, AdaZORO exploits sparsity when it is present, and incurs no penalty when it is not. \\

\paragraph{Boundedness of iterates} Let $\mathcal{X}^{\star} := \{x^{\star}: f(x^{\star}) = \min_{x \in \mathbb{R}^{d}} f(x)\}$ denote the solution set of \eqref{eq:reg_opt_problem} (taking $r = 0$ for simplicity). A common difficulty in analyzing the convergence of many iterative zeroth-order (and first-order) optimization methods is to show the sequence of distances between the iterates $x_k$ and $\mathcal{X}^{\star}$ remains bounded. Typically, this is either assumed directly \citep{berahas2021global} or shown by assuming  (i) $\mathcal{X}^{\star}$ is a singleton ({\em i.e.} $\mathcal{X}^{\star} = \{x^{\star}\}$) \citep{schmidt2011convergence, ghadimi2013stochastic}, or (ii) there exists an $R > 0$ such that $\max_{x^{\star} \in \mathcal{X}^{\star}}\|x - x^{\star}\|_2 \leq R$ for all $x$ in the level set $\{x: f(x) \leq f(x_0)\}$ \citep{stich2013optimization,tappenden2016inexact,bergou2020stochastic}. For convex $f$, this ``bounded level sets'' assumption ({\em i.e.} assumption (ii)) is equivalent to assuming $\mathcal{X}^{\star}$ is compact \citep[Proposition~B.10]{bertsekas1997nonlinear}, in which case (i) is a special case of (ii). However, for many functions exhibiting gradient sparsity or compressibility, $\mathcal{X}^{\star}$ is decidedly non-compact. For example, the sparse quadratic problem (see synthetic dataset I, case (a) in Section~\ref{sec:Synth_Datasets})  has $\mathcal{X}^{\star} = \mathbb{R}^{d-s}$. As such we cannot use the usual tools for showing boundedness of the iterates, and are forced to develop a new approach using an extended notion of coercivity.\\

\paragraph{Empirical validation} ZORO achieves its improved query complexity by exploiting gradient compressibility. Thus, it is natural to question how common this property is in real-world problems. We show empirically that gradients for two common problems, portfolio optimization and adversarial attacks on neural networks, are indeed highly compressible (see Figure~\ref{fig:compressible_gradient}). We show for such problems the theoretical query complexity of ZORO is realized in practice.\\

The rest of the paper is laid out as follows. In the remainder of Section~\ref{sec:Introduction}, we discuss the necessary assumptions and notation, summarize the major contributions of this work, and state our main results. In Section~\ref{sec:GradEstimate}, we provide bounds on the gradient estimate error in ZORO. Sections~\ref{sec:GradDescent} and \ref{sec:ProxGrad} contain our technical results on the convergence rates of inexact (prox-) gradient descent while Section~\ref{sec:Boundedness} discusses the issue of iterate boundedness. Section~\ref{sec:AdaZORO} presents ZORO with Adaptive sampling (AdaZORO) while Section~\ref{sec:numberical} contains the results of our numerical experiments. Finally, in Appendix~\ref{sec:SparseGradients} we clarify several issues regarding functions with sparse gradients that were unclear in the prior literature. 

\subsection{Notation}
For an integer $n$, we define $[n]:= \{1,\ldots, n\}$. For any vector or matrix, $\|\cdot\|_0$ counts the non-zero entries, $\|\cdot\|_1$ is the entry-wise $\ell_1$ norm, and $\|\cdot\|_2$ is the $\ell_2$ norm. We shall frequently use $\bfg(x) := \nabla f(x)$, $\bfg_k := \nabla f(x_k)$ or simply $\bfg$ when the point $x$ in question is clear. We write $F^{\star}:=\min_{x\in\mathbb{R}^{d}}F(x)$ and $f^{\star} = \min_{x\in\mathbb{R}^d}f(x)$, when $r = 0$. Similarly, $e_k := F(x_k) - F^{\star}$ or $e_k = f(x_k) - f^{\star}$ when $r = 0$, where $x_k$ is the $k$-th iterate. By $\left[x \right]_{(s)}$ we mean the best $s$-sparse approximation to $x\in\mathbb{R}^d$:
\begin{equation*}
\left[x \right]_{(s)} = \argmin\{ \|x - v\|_2: \ \|v\|_{0} \leq s \},
\end{equation*}
while $|x|_{(i)}$ denotes the $i$-th largest-in-magnitude component of $x$. We use $\partial r(x)$ to denote the sub-differential of $r$ at $x$, a potentially set-valued operator. As $r$ is convex this is always well-defined. By $\partial F(x)$ we shall mean the limiting sub-differential: $\partial F(x) = \nabla f(x) + \partial r(x)$ (recall $F = f+r$) \citep{mordukhovich2006variational}. A necessary, but not sufficient, condition for $x$ to be a minimizer of $F$ is $0 \in \partial F(x)$ \citep{mordukhovich2006frechet}. If $f$ is convex we say $x_k$ is a $\varepsilon$-optimal solution to \eqref{eq:reg_opt_problem} if $F(x_k) - F^{\star} \leq \varepsilon$. For non-convex $f$ we say $x_k$ is $\varepsilon$-stationary if there exists a $u \in\partial F(x)$ satisfying $\|u\|_2 \leq \varepsilon$. Recall $\mathcal{X}^{\star} := \{x^{\star}: F(x^{\star}) = F^{\star}\}$ denotes the solution set of \eqref{eq:reg_opt_problem}. For convex $f$, and non-empty $\mathcal{X}^{\star}$ define $ \proj_{\star}(x) :=\argmin_{y\in\mathcal{X}^{\star}}\|y-x\|_2$. As $\mathcal{X}^{\star}$ is convex (because $F = f+r$ is) this projection is well defined.

\subsection{Assumptions}
We present a series of assumptions used in this paper.
\begin{assumption}[Sparse/compressible gradients]
\label{assumption:Sparsity}
\begin{enumerate}[label=\ref{assumption:Sparsity}.\alph*]
    \item {(Exact sparsity)}. \label{Assumption:ExactSparsity} The gradients of $f$ are {\em exactly $s_{\mathrm{exact}}$-sparse} if $\|\nabla f(x)\|_{0} \leq s_{\mathrm{exact}}$ for all $x\in\mathbb{R}^d$.
    \item {(Compressibility)}.
    \label{Assumption:Compressibility} The gradients of $f$ are {\em compressible} if there exists a $p \in (0,1)$ such that $|\nabla f(x)|_{(i)} \leq i^{-1/p}\|\nabla f(x)\|_{2}$. 
\end{enumerate}
\end{assumption}

Compressibility does not explicitly specify support size $s$, but for $s\in [d]$ it implies \citep[Section~2.5]{Needell2009}:
\begin{align}
  & \|\nabla f(x) - \left[\nabla f(x)\right]_{(s)}\|_1 \leq \left(1/p - 1\right)^{-1}\|\nabla f(x)\|_2 s^{1 - 1/p} \label{eq:Compressible_ell_1_bound} \\ 
  & \|\nabla f(x) - \left[\nabla f(x)\right]_{(s)}\|_2 \leq \left(2/p - 1\right)^{-1/2}\|\nabla f(x)\|_2 s^{1/2 - 1/p}. \label{eq:Compressible_ell_2_bound}
\end{align}
We also need an assumption which encodes the rapid decay of the sorted entries of the Hessian (see Figure~\ref{fig:compressible_hessian}). We follow \citep{Wang2018} and assume:
\begin{assumption}[Weakly sparse Hessian] \label{assumption:WeakSparsity}
$f$ is twice differentiable and there exists a constant $H$ such that $\|\nabla^{2} f(x)\|_{1} \leq H$ for all $x\in \mathbb{R}^d$. 
\end{assumption}
Figure~\ref{fig:compressible_hessian} suggests we may take $H$ to be small, as $\|\nabla^{2} f(x)\|_{1} := \sum_{i,j}|\nabla_{ij}f(x)|$ and the sorted $|\nabla_{ij}f(x)|$ decay rapidly. It is possible to weaken Assumption~\ref{assumption:WeakSparsity} substantially; the bound on $\|\nabla^{2} f(x)\|_{1}$ need only hold for $x$ in the level set $\{x: f(x) \leq f(x_0)\}$. For ease of exposition we do not do so here.
Next we combine two standard assumptions on smoothness and existence of minimizers.
\begin{assumption}[Solution existence and Lipschitz gradients] \label{assumption:Lipschitz_Diff}
(i) The solution set of $F$ is non-empty.
(ii) $f$ is $L$-Lipschitz differentiable, \textit{i.e.} $\|\nabla f(x) - \nabla f(y)\|_2 \leq L\|x-y\|_2$ for any $x,y\in\mathbb{R}^d$. 
\end{assumption}
We are not assuming access to $\nabla f$, only that the Lipschitz property holds. 
\begin{assumption}[Adversarially noisy oracle] \label{assumption:noise model} We only have access to $f(x)$ through a noisy zeroth-order oracle: $E_{f}(x) = f(x) + \xi$ with $|\xi| \leq \sigma$ for all $x \in \mathbb{R}^{d}$.
\end{assumption}

\begin{remark}
In the stochastic optimization literature ({\em e.g.} \citep{ghadimi2013stochastic,Balasubramanian2018}), it is common to assume $E_{f}(x) = \tilde{f}(x;w)$ where $w$ is a random variable and $f(x) = \mathbb{E}_{w}[\tilde{f}(x;w)]$ while placing a bound on the second moment of the gradient: $\mathbb{E}_{w}[\|\nabla f(x) - \nabla \tilde{f}(x;w)\|_2^2] \leq \sigma^2$. Although our bounded noise model is a somewhat stronger assumption, it allows us to consider noise which is not zero-mean --- the so-called adversarial noise model as it allows for an adversary that chooses each perturbation $\xi$ maliciously.  
\end{remark}
The final two assumptions prescribe a growth rate on $\|\nabla f(x)\|_2$. 
\begin{assumption}[Restricted strong convexity] \label{def:res_convex}
$h$ (\textit{i.e.} either $f$ or $F$) is \emph{restricted $\nu$-strongly convex}, \textit{i.e.} for all $x\in \mathbb{R}^d$:
\begin{align}\label{eq:rsc}
  h(x) - \min h \ge \frac{\nu}{2}\|x - \proj_{\star}(x)\|_2^2.
\end{align}
\end{assumption}

Assumption~\ref{def:res_convex} is a weaker assumption than strong convexity; see \citep{schopfer2016linear,Zhang2017} for more results. We also introduce the following extended notions of coercivity:

\begin{assumption}[Coercivity]
\begin{enumerate}[label=\ref{assumption:Coercivity}.\alph*]
\item $f$ is {\em coercive} if for any $\{x_k\}_{k=1}^{\infty}$ satisfying $\lim_{k\to\infty}\|x_k - \proj_{\star}(x_k)\|_2 = +\infty$, we also have $\lim_{k\to\infty}f(x_k)\to\infty$. \label{assumption:f_coercive}
\item $\nabla f$ is coercive with respect to $f$ if for any $\{x_k\}_{k=1}^{\infty}$ satisfying $\displaystyle \lim_{k\to\infty}f(x_k) = + \infty$, we also have  $\displaystyle \lim_{k\to\infty}\|\nabla f(x_k)\|_2 = + \infty$. \label{assumption:grad_coercive}
\item $\partial F$ is coercive with respect to $F$ if for any $\{x_k\}_{k=1}^{\infty}$ satisfying $\displaystyle \lim_{k\to\infty}F(x_k) = + \infty$, we also have $\displaystyle \lim_{k\to\infty}\inf_{u\in \partial F(x_{k})}\|u\|_2 = +\infty$. \label{assumption:subgrad_coercive}
\end{enumerate}
\label{assumption:Coercivity}
\end{assumption}

\subsection{Prior work}
Many approaches to zeroth-order optimization use the following template:
\begin{enumerate}
    \item Construct an estimator $\hat{\bfg}_k$ of $\bfg_k := \nabla f(x_k)$.
    \item Take a negative gradient step $x_{k+1} = x_{k} - \alpha_{k}\hat{\bfg}_k$.
\end{enumerate}
The finite difference approach: $\hat{\bfg}_k := \sum_{i=1}^{d}\frac{f(x_k + \delta\bfe_i) - f(x_k)}{\delta}\bfe_i$, where $\delta > 0$ is a sampling radius and $\bfe_i$ denotes the $i$-th canonical basis vector was introduced as FDSA \citep{Kiefer1952}. This results in an accurate estimator, but requires $d+1$ queries per iteration, and thus is not query-efficient. To overcome this, randomized estimators were employed in SPSA \citep{spall1998overview} and Random Search \citep{nesterov2011random,Nesterov2017} that use only two queries per iteration. We mention also the work of \citep{flaxman2004online}, which requires only one query per iteration (but at the cost of a slower convergence rate) and the coordinate-descent-style approaches of \citep{stich2013optimization, kim2021curvature}. Variance reduced versions of Random Search \citep{chen2019zo,mania2018simple,salimans2017evolution,Balasubramanian2018}  which use $2 <m<d$ queries to produce a lower variance estimator yield empirically better performance, but achieve the same asymptotic rate of convergence as Random Search. Recently, several works \citep{Wang2018,choromanski2018structured,cai2020one, berahas2021theoretical,cai2021zeroth} have considered finite differences $y_{i} = \frac{f(x_k + \delta \bfz_i) - f(x_k)}{\delta}$ as noisy approximations to the directional derivatives $\bfz_{i}^{\top}\bfg_{k}$ and investigated various regression schemes for recovering $\bfg_k$ from these linear measurements. We discuss the relationship between this line of work and our own in Section~\ref{sec:GradEstimate}. 

As querying the oracle is typically expensive, the most important metric for comparing zeroth-order optimization algorithms is their {\em query complexity}, defined as the number of queries required to find an iterate $x_{k}$ such that $f(x_{k}) - f^{\star} \leq \varepsilon$. Here, there are two different approaches to the analysis. One can assume that the oracle noise, $\xi$, is {\em zero mean}, in which case arbitrarily small $\varepsilon$ is possible. If $\xi$ is not zero mean, there is a lower bound on $\varepsilon$ stemming from the fact that when the magnitude of the gradient is of the same order as the noise no further progress can be made. The former approach yields higher complexity; \citep{ghadimi2013stochastic} showed if $f$ is Lipschitz differentiable and convex then Random Search finds an $\varepsilon$-optimal solution in $\cO(d/\varepsilon^{2})$ queries, while  \citep{jamieson2012query} showed {\em any} algorithm for this problem necessarily requires $\Omega(d/\varepsilon^2)$ queries. Faster rates are achievable in the latter approach, but only for $\varepsilon$ lower bounded by a constant depending on the noise level. Common to both approaches is a polynomial dependence of query complexity on $d$. 

In order to break this unfortunate dependence on $d$, \citep{Wang2018} and \citep{Balasubramanian2018} assume exact gradient sparsity (Assumption~\ref{Assumption:ExactSparsity}). Specifically, \citep{Wang2018} uses LASSO to construct $\hat{\bfg}_k$ and assumes zero mean noise to achieve a query complexity of $O\left(s\log^3(d)/\varepsilon^3\right)$. \citep{Balasubramanian2018} claims that gradient descent, using the Random Search estimator $\hat{\bfg}_k$, benefits from {\em implicit regularization} and automatically achieves a query complexity of $O\left(s\log^2(d)/\varepsilon^2\right)$, assuming zero mean noise, as long as the step size is carefully chosen. Unfortunately, their analysis is flawed and only holds when the support of $\nabla f(x)$ is {\em the same for all} $x\in\mathbb{R}^{d}$. We discuss this further in Appendix~\ref{sec:SparseGradients}. Neither of these works considers compressible gradients, (Assumption~\ref{Assumption:Compressibility}), non zero-mean noise or any notion of strong convexity. 

Finally, we note the many works \citep{luo1993error,blatt2007convergent,schmidt2011convergence,friedlander2012hybrid,nedic2010effect,berahas2021global} that study gradient descent: $x_{k+1} = x_{k} - \alpha \hat{\bfg}_k$, where $\hat{\bfg}_k$ is a biased estimator of the true gradient $\bfg_k$. As any estimator derived from zeroth-order queries is necessarily biased, these results are closely connected to the convergence analysis of zeroth-order methods. We discuss the relationship between these results and our own in Sections~\ref{sec:GradDescent} and \ref{sec:ProxGrad}.

\subsection{Contributions} We summarize the contributions of this paper.
\begin{enumerate}[(i)]
    \item We introduce the idea of gradient compressibility to zeroth-order optimization.
    \item We propose an algorithm, ZORO, which exploits gradient compressibility. 
    \item We show theoretically ZORO has a query complexity only logarithmically dependent on the extrinsic dimension $d$. Proving this requires overcoming a number of technical challenges, particularly analyzing inexact prox-gradient descent with constant gradient error and proving the iterates $x_k$ remain bounded even though the level sets of $f$ are not bounded.
    \item We propose a heuristic improvement to ZORO, called AdaZORO, which dynamically adapts to varying levels of gradient compressibility.
    \item We provide empirical evidence that {\em gradient compressibility} occurs in real-world applications. We also show numerically ZORO (and AdaZORO) can successfully exploit this gradient compressibility. 
\end{enumerate}

\begin{algorithm}[tb] 
   \caption{Zeroth-Order Regularized Optimization Method (ZORO)} \label{algo:zoro}
\begin{algorithmic}[1]
   \State {\bfseries Input:} $x_0$: initial point; $s$: gradient sparsity level; $\alpha$: step size; $\delta$: query radius, $K$: number of iterations. 
   \State $m \gets b_1 s\log(d/s)$ \quad {where $b_1$ is as in Theorem~\ref{thm:SatisfiesRIP}. Typically, $b_1 \approx 1$ is appropriate}
   \State $z_1,\dots,z_m\gets$ i.i.d. Rademacher random vectors
   \For{$k=0$ {\bfseries to} $K$}
        \For{$i=1$ {\bfseries to} $m$}
            \State $y_{i} \gets (E_f(x+\delta z_i)-E_f(x))/\delta$
        \EndFor
        \State $\bfy\gets \frac{1}{\sqrt{m}}[y_1,\ldots, y_m]^{\top}$
       \State $Z\gets\frac{1}{\sqrt{m}}[z_1,\ldots, z_m]^{\top}$
       \State ${\hat{\bfg}}_k \approx \argmin_{\|\mathbf{g}\|_{0} \leq s}\|Z\mathbf{g} - \mathbf{y}\|_2 \quad $ by CoSaMP
       \State $x_{k+1}\gets\mathbf{prox}_{\alpha r}(x_k-\alpha \hat{\bfg}_k)$
   \EndFor
   \State {\bfseries Output:} $x_K$: minimizer of \eqref{eq:reg_opt_problem}.
\end{algorithmic}
\end{algorithm}

\subsection{Main results}
Our first result is for the non-regularized case, but allows for compressible gradients. Theorems~\ref{thm:Main_compressible} and \ref{thm:Main_regularized_sparse} (and Lemma~\ref{thm:Rel_Error_GD} and Theorem~\ref{thm:Convergence_Abs_Error_Reg}) depend on a constant $R$ satisfying $\|x_k - \proj_{\star}(x_k)\|_2 \leq R$ for all $k$.This is analogous to the constant $D$ in \citep[Assumption~4.5]{berahas2021global}, the constant $R_0$ in \citep[Assumption~5.1]{bergou2020stochastic}, the ``level set radius'' $\mathcal{R}_w(x_0)$ in \citep{tappenden2016inexact} or the diameter of the feasible set $B$ in \citep[Assumption~2]{Wang2018}. Indeed, in the special case where $r$ is the indicator function of a compact convex set $\mathcal{X}$ one can simply take $R = \text{diam}(\mathcal{X}) := \max_{x,y\in\mathcal{X}} \|x-y\|_2$. As discussed in Section~\ref{sec:Introduction}, we cannot use the bounded level sets assumption because many $f$ exhibiting sparse gradients do not have this property. In Section~\ref{sec:Boundedness}, we deduce the existence of such an $R$ from coercivity properties of $f$ ({\em i.e.} Assumption~\ref{assumption:Coercivity}).

\begin{theorem}[No regularizer, compressible gradients]
\label{thm:Main_compressible}
Suppose $f$ is convex and satisfies Assumptions~\ref{Assumption:Compressibility}, \ref{assumption:WeakSparsity}, \ref{assumption:Lipschitz_Diff}, \ref{assumption:f_coercive} and \ref{assumption:grad_coercive}. Choose $s$ large enough so $\psi := b_4s^{1/2-1/p} \leq 0.35$ and choose $\alpha = \frac{1}{L}$. Then ZORO finds an $\varepsilon$-optimal solution in $\displaystyle \frac{4b_1s\log(d)LR^{2}}{\varepsilon(1-8\psi^{2})}$ queries for any $\varepsilon > b_3R\sqrt{2\sigma H/(1-8\psi^{2})}$. If instead of Assumptions~\ref{assumption:f_coercive} and \ref{assumption:grad_coercive}, $f$ satisfies Assumption~\ref{def:res_convex}, then this query complexity improves to:
\begin{equation*}
\frac{b_1s\log(d)\log\left(\frac{\varepsilon}{e_0} - \frac{2b_3^2\sigma H}{\nu e_0(1-8\psi^{2})}\right)}{\log\left(1 - \frac{(1-8\psi^{2})\nu}{4L}\right)} = O\left(s\log(d)\log\left(\frac{1}{\varepsilon}\right)\right).
\end{equation*}
Both query complexities hold with probability $1 - 2(s/d)^{b_2 s}$.
\end{theorem}

Our second result allows for $r \neq 0$. Due to technical difficulties, we only prove this result for sparse gradients. Empirically, we have observed excellent performance of ZORO with regularizer for $f$ having merely compressible gradients (see Section~\ref{subsec:asset risk}).

\begin{theorem}[Regularized, sparse gradients]
\label{thm:Main_regularized_sparse}
Suppose $f$ is convex and satisfies Assumptions~\ref{Assumption:ExactSparsity}, \ref{assumption:WeakSparsity}--\ref{assumption:noise model}. Suppose $r$ is convex and $F = f + r$ satisfies Assumptions~\ref{assumption:f_coercive} and \ref{assumption:subgrad_coercive}. Choose $s \geq s_{\mathrm{exact}}$ and $\alpha = \frac{1}{L}$. Then ZORO finds an $\varepsilon$-optimal solution in
$\displaystyle b_1s\log(d)\left(\frac{36LR^{2}}{\varepsilon} + \log(e_0)\right)$ queries, for any $\varepsilon \geq b_5R\sqrt{\sigma H}$. If instead of Assumptions~\ref{assumption:f_coercive} and \ref{assumption:subgrad_coercive} $F$ satisfies Assumption~\ref{def:res_convex}, this query complexity improves to:
\begin{equation*}
\frac{b_1s\log(d)\log\left(\frac{\varepsilon\nu - b_{5}^{2}\sigma H}{\nu e_0}\right)}{\log\left(24L\right) - \log\left(\nu + 24L\nu\right)} = O\left(s\log(d)\log\left(\frac{1}{\varepsilon}\right)\right)
\end{equation*}
for any $\varepsilon > b_5^{2}\sigma H$. Again, both query complexities hold with probability $1 - 2(s/d)^{b_2 s}$.
\end{theorem}

\begin{remark}
The constants $b_1$--$b_5$ arise from the use of compressed sensing to reconstruct $\bfg_k$. They depend on the particular algorithm used (we use CoSaMP) and the number of iterations this algorithm is run for. They do not depend on $F$ or $\varepsilon$.
\end{remark}

We also provide the following convergence-to-stationarity theorem for non-convex $f$, which {\em does not} require any coercivity assumptions

\begin{theorem}
\label{thm:MainNonConvex}
Suppose $f$ satisfies Assumptions~\ref{Assumption:ExactSparsity} and \ref{assumption:WeakSparsity}--\ref{assumption:Lipschitz_Diff} while $r$ is convex and proximable. ZORO finds an $\varepsilon$-stationary solution in $\displaystyle \frac{12b_1L e_0 s\log(d)}{\left(\varepsilon - b_5\sqrt{\sigma H}\right)^{2}}$ queries, for any $\varepsilon > b_5\sqrt{\sigma H}$, with probability $1 - 2(s/d)^{b_2 s}$.
\end{theorem}
 
\section{Estimating the gradient}
\label{sec:GradEstimate}
Choose the number of queries $m$ and a sampling radius $\delta >0$, and let $\{z_i\}_{i=1}^m \subset \mathbb{R}^d$ be Rademacher random vectors ({\em i.e.} $(z_{i})_j=\pm 1$ with equal probability for $j=1,\ldots,d$). Other types of random vectors certainly work too, but for conceptual clarity we restrict to Rademacher. Each measurement is:
\begin{equation} \label{eq:y_i}
y_{i} = \frac{1}{\sqrt{m}}\frac{E_f(x+\delta z_i)-E_f(x)}{\delta}. 
\end{equation}
As in \citep{Wang2018,choromanski2020provably} we think of the $y_i$ as noisy approximations to directional derivatives:

\begin{lemma}  \label{lm:query bound}
  If Assumptions~\ref{assumption:WeakSparsity} and \ref{assumption:noise model} are satisfied, then 
 \begin{equation*}
 y_i = \frac{1}{\sqrt{m}}z_i^{\top} \bfg + \frac{\mu_i}{\delta} + \delta \nu_i
 \label{eq:y_i_with_error_terms_2}
 \end{equation*}
 with $\bfg = \nabla f(x)$, $|\mu_i| \leq 2\sigma/\sqrt{m}$, and $|\nu_i| \leq H/(2\sqrt{m})$.
 \end{lemma}
\begin{proof}[Proof of Lemma~\ref{lm:query bound}]
The proof is similar to the argument of Section 3 in \citep{Wang2018}, but we include it for completeness. From Taylor's theorem, for some $t\in(0,1)$:
 \begin{align*}
 f(x+\delta z_i) & = f(x) + \delta z_i^{\top}\bfg + \frac{\delta^{2}}{2}z_{i}^{\top}\nabla^{2}f(x+tz_i)z_i.
 \end{align*}
Writing $E_{f}(x+\delta z_i) = f(x+\delta z_i) + \xi_{+}$ and $E_{f}(x) = f(x) + \xi_{-}$ \eqref{eq:y_i} becomes:
 \begin{align*}
 y_i = \frac{1}{\sqrt{m}}z_{i}^{\top}\bfg + \frac{\xi_{+} - \xi_{-}}{\sqrt{m}\delta} + \frac{\delta}{2\sqrt{m}}z_{i}^{\top}\nabla^{2}f(x+tz_i)z_i .
 \end{align*}
Let $\mu_i := \frac{\xi_{+} - \xi_{-}}{\sqrt{m}}$, then $|\mu_i| \leq 2\sigma/\sqrt{m}$. Let $\nu_{i} :=z_{i}^{\top}\nabla^{2}f(x+tz_i)z_i/(2\sqrt{m})$. Now:
\begin{align*}
     2\sqrt{m}|\nu_{i}| & = \left|z_{i}^{\top}\nabla^{2}f(x + t\delta z_i)z_{i}\right| \\
                & = \bigg|\sum_{j,k} \nabla^{2}_{j,k}f(x + t\delta z_i)(z_i)_{j}(z_i)_{k}\bigg| \\
                & \leq \|\nabla^{2}f(x + t\delta z_i)\|_{1}\|z_i\|_\infty^2 \stackrel{(a)}{\leq} H
 \end{align*}
 where $(a)$ follows from Assumption~\ref{assumption:WeakSparsity} and $\|z_i\|_{\infty} = 1$.
\end{proof}
Let $\bfy = [y_1,\ldots, y_m]^{\top}$, $\boldsymbol{\mu} = [\mu_1,\ldots, \mu_{m}]^{\top}$ and $\boldsymbol{\nu} = [\nu_1,\ldots, \nu_{m}]^{\top}$. Define $Z\in\mathbb{R}^{m\times d}$ to be the {\em sensing matrix} whose $i$-th row is $\frac{1}{\sqrt{m}}z_{i}^{\top}$. Then:
\begin{equation} \label{eq:Expanded_Equation}
    \bfy = Z\bfg + \frac{1}{\delta}\boldsymbol{\mu} + \delta\boldsymbol{\nu}.
\end{equation}
Several recent works attempt to recover $\bfg$ from \eqref{eq:Expanded_Equation}. \citep{berahas2021theoretical} considers taking $m = d$ measurements and solving the linear system, while \citep{Wang2018} assumes $\bfg$ is {\em exactly sparse} and solves the LASSO problem\footnote{Their approach is slightly different as they approximate $\bfg$ {\em and} $f(x)$ using the same LASSO problem}:
\begin{equation}
\hat{\bfg} = \argmin \|Z\bfv - \bfy\|_{2}^{2} + \lambda \|\bfv\|_{1}.
\label{eq:LassoForGradient}
\end{equation}
In \citep{choromanski2020provably}, recovering $\bfg$ by solving the more general regularized regression problem:
\begin{equation*}
\hat{\bfg} = \argmin \|Z\bfv - \bfy\|_{p}^{p} + \alpha \|\bfv\|_{q}
\end{equation*}
is proposed, and in the special case $p=1,\alpha=0$ ({\em i.e.} LP decoding) bounds on $\|\bfg - \hat{\bfg}\|_2$ are proved which allow for an extraordinary amount of noise, but require $m = \Omega(d)$. In this work, we approximate $\bfg$ by using a {\em greedy approach} on the {\em nonconvex} problem:
\begin{equation} \label{eq:SparseRecoveryForGradient}
    \hat{\bfg} = \argmin\nolimits_{\bfv\in\mathbb{R}^{d}}\|Z\bfv - \mathbf{y}\|_2 \quad \text{ such that } \|\mathbf{v}\|_{0} \leq s.
\end{equation}
Let us briefly mention several advantages this approach enjoys over prior work:
\begin{enumerate}[(i)]
    \item Unlike the LP decoding approach of \citep{choromanski2020provably}, solving \eqref{eq:SparseRecoveryForGradient} exploits sparsity or compressibility to reduce the number of samples, $m$, from $\Omega(d)$ to $\cO(s\log(d))$. 
    \item When $d$ is large and $s$ is small, solving \eqref{eq:SparseRecoveryForGradient} using a good algorithm such as CoSaMP can be significantly faster than solving \eqref{eq:LassoForGradient}. (See Figure~\ref{fig_time_all})
    \item The LASSO estimator is typically biased \citep{fan2001variable}, while the estimator arising from \eqref{eq:SparseRecoveryForGradient} does not have this problem. This creates an additional source of error for LASSO when $\bfg$ is merely {\em compressible} instead of exactly sparse.
    \item Empirically, we have found that unless $\lambda$ in \eqref{eq:LassoForGradient} is adjusted as $\|\bfg_k\|_2$ decreases, the quality of the estimator $\hat{\bfg}_k$ degrades until gradient descent with $\hat{\bfg}_k$ no longer makes progress (see Figure~\ref{fig:Algo_Comparisons}). It is unclear how to choose $\lambda$ in a principled manner, without {\em a priori} knowledge of $\|\bfg_k\|_{2}$.
\end{enumerate}
We shall use CoSaMP \citep{Needell2009} for \eqref{eq:SparseRecoveryForGradient}. We emphasize that $\hat{\bfg}$ is a sparse approximation to the true gradient $\bfg$. When $\bfg$ is sparse or compressible, this approximation is highly accurate. When $\bfg$ is neither sparse nor compressible, $\hat{\bfg}$ is still likely to be a descent direction, and thus can still be used within a gradient descent scheme.

\subsection{Analysis of CoSaMP for gradient estimation} $Z$ has the $4s$-Restricted Isometry Property ($(4s)$-RIP) if, for all $\bfv\in\mathbb{R}^{d}$ with $\|\bfv\|_{0} \leq 4s$:
 \begin{equation*}  \label{eq:RIP_Definition}
     (1 - \delta_{4s}(Z))\|\bfv\|_2^2 \leq \|Z\bfv\|_2^2 \leq (1 + \delta_{4s}(Z))\|\bfv\|_2^2.
 \end{equation*}
for some $\delta_{4s}(Z) \in (0,1)$. If $m$ is proportional to $s\log(d)$ then $Z$ as constructed above will have the $4s$-RIP almost surely:
 
\begin{theorem}[Theorem 5.2 of \citep{Baraniuk2008}]  \label{thm:SatisfiesRIP}
If $m = b_1 s\log(d/s)$, then $Z$ has the $4s$-RIP with $\delta_{4s}(Z)\leq 0.3843$ with probability $1 - 2(s/d)^{b_2 s}$. Here $b_1$ and $b_2$ are constants independent of $s,d$ and $m$. 
\end{theorem}

 The choice of $0.3843$ is to match with the assumptions of \citep{Foucart2012}, the main result of which we state next. Recall $\left[\bfg\right]_{(s)}$ denotes the best $s$-sparse approximation to $\bfg$.
 
\begin{theorem}[Theorem 5 of \citep{Foucart2012}]
Let $\{\bfg^n\}$ be the sequence generated by applying CoSaMP \citep{Needell2009} to problem \eqref{eq:SparseRecoveryForGradient}, with $m \geq b_1 s\log(d/s)$ and initialization $\bfg^{0} = \mathbf{0}$. Then, with probability $1 - 2(s/d)^{b_2 s}$: 
\begin{equation}
    \|\bfg^{n} - \bfg\|_2 \leq \|\bfg - [\bfg]_{(s)}\|_2 + \tau \|Z(\bfg - \left[\bfg\right]_{(s)})\|_2 + \frac{\tau}{\delta}\|\boldsymbol{\mu}\|_2 + \tau\delta\|\boldsymbol{\nu}\|_2  + \rho^{n}\|\left[\bfg\right]_{(s)}\|_2,
    \label{eq:Break_into_Three}
\end{equation}
for all $\bfg \in \mathbb{R}^{d}$, where $\rho < 1$ and $\tau\approx 10$ depend only on $\delta_{4s}$.
\label{thm:CoSaMPErrorBounds}
 \end{theorem}
We emphasize this result is {\em universal}, {\em i.e.} it holds for all $\bfg \in \mathbb{R}^{d}$ with the stated probability. The exact values of $\rho$ and $\tau$ are provided in \citep{Foucart2012}. One can make $\tau$ smaller by making $m$ larger \citep{Foucart2012}. The constants $b_1, b_2$ are the same as in Theorem~\ref{thm:SatisfiesRIP}. Other initializations are possible; for example, we have found using $\bfg^{0} = \hat{\bfg}_{k-1}$ at the $k$-th iteration offers a modest speedup.

\begin{theorem}
Suppose $\bfg$ is compressible (Assumption~\ref{Assumption:Compressibility}) and, for any $s\in [d]$, $Z$ is chosen according to Theorem~\ref{thm:SatisfiesRIP}. Then $\|\bfg - [\bfg]_{(s)}\|_2 + \tau\|Z\left(\bfg - [\bfg]_{(s)}\right)\|_2 \leq \psi \|\bfg\|_2$, where:
\begin{equation*}
\psi = \underbrace{\left(\left(1 + \tau\sqrt{1 + \delta_{4s}(Z)}\right)\left(\frac{2}{p}-1\right)^{-1/2} + \tau\sqrt{1 + \delta_{4s}(Z)}\left(\frac{1}{p}-1\right)^{-1}\right)}_{= b_4} s^{1/2-1/p}.
\end{equation*}
\label{Lemma:Compressible_Tail_Bound}
\end{theorem}

\begin{proof}
From \citep{Needell2009} if $Z$ satisfies the $(4s)$-RIP then for any $\bfv\in\mathbb{R}^{d}$:
\begin{equation*}
\|Z\bfv\|_2 \leq \sqrt{1 + \delta_{4s}(Z)}\left(\|\bfv\|_2 + \frac{1}{\sqrt{s}}\|\bfv\|_1\right).
\end{equation*}
Combining this with \eqref{eq:Compressible_ell_1_bound} and \eqref{eq:Compressible_ell_2_bound}:
\begin{align*}
 &\quad~\|Z\left(\bfg - [\bfg]_{(s)}\right)\|_2 \\
 &\leq \sqrt{1 + \delta_{4s}(Z)}\left(\|\bfg - [\bfg]_{(s)}\|_2 + \frac{1}{\sqrt{s}}\|\bfg - [\bfg]_{(s)}\|_1\right) \\
    & \leq \sqrt{1 + \delta_{4s}(Z)}\left(\left(\frac{2}{p}-1\right)^{-1/2}\|\bfg\|_2s^{1/2-1/p} + \frac{1}{\sqrt{s}}\left(\frac{1}{p} - 1\right)^{-1}\|\bfg\|_2s^{1-1/p}\right) \\
    & = \sqrt{1 + \delta_{4s}(Z)}\left(\left(\frac{2}{p}-1\right)^{-1/2} + \left(\frac{1}{p}-1\right)^{-1}\right)s^{1/2-1/p}\|\bfg\|_2.
\end{align*}
Use \eqref{eq:Compressible_ell_2_bound} again to bound $\|\bfg - [\bfg]_{(s)}\|_2$ and add to obtain the lemma.
\end{proof}

We now bound the error terms in our measurements:
\begin{lemma} \label{lemma:Bound_chi_xi_vec}
$\|\boldsymbol{\mu}\|_{2} \leq 2\sigma$ and $\|\boldsymbol{\nu}\|_2 \leq H/2$.
\end{lemma}
 
 \begin{proof}
    From Lemma~\ref{lm:query bound},
    \begin{equation*}
        \|\boldsymbol{\mu}\|_2^2  = \sum_{i=1}^{m} \mu_i^2 \leq \sum_{i=1}^{m}\frac{4\sigma^2}{m}  = 4\sigma^2.
    \end{equation*}
    Similarly,
    \begin{equation*}
        \|\boldsymbol{\nu}\|_2^2  = \sum_{i=1}^{m} \nu_i^2 \leq \sum_{i=1}^{m}\frac{H^2}{4m} = \frac{H^2}{4}.
    \end{equation*}
 \end{proof}

Combining Theorem~\ref{Lemma:Compressible_Tail_Bound}, Lemma~\ref{lemma:Bound_chi_xi_vec} and \eqref{eq:Break_into_Three}; and using $\|[\bfg]_{(s)}\|_2 \leq \|\bfg\|_2$ yields:

\begin{theorem}  \label{thm:Grad_Estimate_Error}
    Suppose $f$ satisfies Assumptions~\ref{assumption:Lipschitz_Diff} and \ref{assumption:Sparsity}. Let $\hat{\bfg}_k$ be the output of Line 10 of Algorithm~\ref{algo:zoro}. Then the error bound
    \begin{equation}        \label{eq:ErrorBound2}
       \left\|\hat{\bfg}_k - \bfg_k\right\|_{2} \leq \left( \psi + \rho^{n}\right)\|\bfg_k\|_2 + \frac{2\tau\sigma}{\delta} + \frac{\tau\delta H}{2}
    \end{equation}
    holds for all $k$, with probability at least $1 - 2(s/d)^{b_2 s}$, where $\rho,\tau,b_1$ and $b_2$ are fixed numerical constants, and:
    \begin{equation*}
    \psi = 
    \begin{cases}
    0, & \textnormal{ if $f$ satisfies Assumption~\ref{Assumption:ExactSparsity} and $s \geq s_{\mathrm{exact}}$}
    \\
    b_4s^{1/2-1/p},    & \textnormal{ if $f$ satisfies Assumption~\ref{Assumption:Compressibility}, for any $s$}
    \end{cases}.
    \end{equation*} 
    Note that $b_4$ depends on $p$ and $\{z_i\}_{i=1}^{m}$, but not on $s$.
 \end{theorem}
 
If $\sigma = 0$, {\em i.e.} the oracle is noise-free, the second term on the right-hand side of \eqref{eq:ErrorBound2} drops out and one can make the third term arbitrarily small by choosing the sampling radius $\delta$ sufficiently small. If $\sigma >0$ then there is a lower bound to how small we can make the right-hand side of \eqref{eq:ErrorBound2}:

\begin{corollary}
\label{cor:ErrorBound3}
 Suppose $\sigma >0$ and that the other assumptions are as in Theorem~\ref{thm:Grad_Estimate_Error}. Choosing $\delta = 2\sqrt{\sigma/H}$ provides the tightest possible error bound of:
 \begin{equation*}
 \left\|\hat{\bfg} - \bfg\right\|_{2} \leq \left( \psi + \rho^{n}\right)\|\bfg\|_2 +2\tau\sqrt{\sigma H}.
 \end{equation*}
\end{corollary}
\begin{proof}
This follows by minimizing $\frac{2\tau\sigma}{\delta} + \frac{\tau\delta H}{2}$ with respect to $\delta$. 
\end{proof}

\section{Gradient descent with relative and absolute errors}
\label{sec:GradDescent}
We consider the problem of minimizing $f$ using inexact gradient descent: $x_{k+1} = x_{k} -\alpha\hat{\bfg}_{k}$, where
\begin{equation}
\|\bfg_{k} - \hat{\bfg}_{k}\|_2^2 \leq \varepsilon_{\mathrm{rel}}\|\bfg_k\|_2^2 + \varepsilon_{\mathrm{abs}}.
\label{eq:Abs_Rel_Error}
\end{equation}
\begin{lemma}
\label{thm:Rel_Error_GD}
Suppose $f$ is convex and satisfies Assumptions~\ref{assumption:Lipschitz_Diff}, \ref{assumption:f_coercive} and \ref{assumption:grad_coercive}. Suppose for all $k$ $\hat{\bfg}_k$ satisfies \eqref{eq:Abs_Rel_Error} with $\varepsilon_{\mathrm{rel}} < 1$. Choose $\alpha = \frac{1}{L}$. Then:
\begin{equation*}
e_{k} \leq \max\left\{\frac{4LR^{2}e_0}{(1-\varepsilon_{\mathrm{rel}})e_0k + 4LR^{2}}, R\sqrt{\frac{ 2\varepsilon_{\mathrm{abs}}}{1-\varepsilon_{\mathrm{rel}}}}\right\}.
\end{equation*}
where $R$ is a constant satisfying $\|x_k - \proj_{\star}(x_k)\|_2 \leq R$ for all $k$ determined by the coercivity conditions (\ref{assumption:f_coercive} and \ref{assumption:grad_coercive}). If instead of Assumptions~\ref{assumption:f_coercive} and \ref{assumption:grad_coercive} we assume $f$ satisfies Assumption~\ref{def:res_convex} then this rate improves to:
\begin{equation*}
e_{k} \leq \left(1 - \frac{(1-\varepsilon_{\mathrm{rel}})\nu}{4L}\right)^{k} e_{0} + \frac{2\varepsilon_{\mathrm{abs}}}{\nu(1 - \varepsilon_{\mathrm{rel}})}.
\end{equation*}
\end{lemma}

Various forms of this result are well-known (\citep{berahas2021global,friedlander2012hybrid} and \citep[Section~1.2]{bertsekas1997nonlinear}), but we were unable to find a precise statement in the literature allowing for non-decreasing $\varepsilon_{\mathrm{abs}}$ and $\varepsilon_{\mathrm{rel}}$ or restricted strongly convex $f$. Hence, we provide a proof in Section~\ref{sec:Proving_Rel_Error_GD}.

\subsection{Deducing Theorem~\ref{thm:Main_compressible} from Lemma~\ref{thm:Rel_Error_GD}}
Before proving Lemma~\ref{thm:Rel_Error_GD} let us explain how Theorem~\ref{thm:Main_compressible} will follow from it. Squaring Corollary~\ref{cor:ErrorBound3}:
\begin{equation*}
\left\|\bfg_{k} - \hat{\bfg}_{k}\right\|_{2}^{2} \leq 2\left( \psi + \rho^{n}\right)^{2}\|\bfg_{k}\|_2^2 +4\tau^{2}\sigma H \quad \text{ for all $k$}
\end{equation*}
with probability $1 - 2(s/d)^{b_2s}$. Choose $n$, the number of iterations of CoSaMP performed, large enough so $\rho^{n} < \psi$, in which case $\varepsilon_{\mathrm{rel}} := 2\left( \psi + \rho^{n}\right)^{2} = 8\psi^{2} < 1$ as long as $\psi< 0.35$. Note that $\varepsilon_{\mathrm{abs}} := 4\tau^{2}\sigma H$. For any $\varepsilon > R\sqrt{2\varepsilon_{\mathrm{abs}}/(1-\varepsilon_{\mathrm{rel}})} =: b_3R\sqrt{2\sigma H/(1-8\psi^{2})}$ we may solve for $k$ guaranteeing $e_{k} \leq \varepsilon$:
\begin{equation*}
k \geq \frac{4LR^2}{\varepsilon(1- 8\psi^{2})} \geq \frac{4LR^2}{\varepsilon(1- 8\psi^{2})}  - \frac{4LR^2}{e_0(1-8\psi^{2})}.
\end{equation*}
Recall ZORO makes $b_1s\log(d)$ queries per iteration. Multiplying this number by the number of required iterations ({\em i.e.} $k$) yields the first result. For strongly convex $f$ by the same line of reasoning for any $\varepsilon > \frac{2\varepsilon_{\mathrm{abs}}}{\nu(1-\varepsilon_{\mathrm{rel}})} = \frac{2b_3^2\sigma H}{\nu(1-8\psi^2)}$ we may guarantee $e_k \leq \varepsilon$ as long as:
\begin{equation*}
k \geq \frac{\log\left(\frac{\varepsilon}{e_0} - \frac{2b_3^2\sigma H}{\nu e_0(1-8\psi^{2})}\right)}{\log\left(1 - \frac{(1-8\psi^{2})\nu}{4L}\right)}.
\end{equation*}
Again, multiplying by the number of queries per iteration yields the result. 

\subsection{Proof of Lemma~\ref{thm:Rel_Error_GD}}
\label{sec:Proving_Rel_Error_GD}
First, we need two lemmas:

\begin{lemma}[Sequence analysis I]\label{prop:seq}
Consider a sequence $\{e_k\}_{k=0}^{\infty}$ with $e_k\ge 0$ and $e_{k+1}\le e_k - c e_k^2+d$ for all $k$, where $c >0 $ and $d\geq 0$. If $d > 0$ we have
\begin{align*}
e_k \le \frac{2e_0}{ce_0k +2},\quad k\in \{t:e_{0},\dots,e_{t+1} \ge \sqrt{2d/c}\},
\end{align*}
while if $d = 0$ we have $e_{k} \leq \frac{e_0}{ce_0k+1}$ for all $k$. 
\end{lemma}
\begin{proof}
If $e_{k} \ge \sqrt{2d/c}$, then $e_{k+1}\le e_k - d$, so $\frac{e_k}{e_{k+1}}\ge 1$. Dividing the condition by $e_{k+1}e_k$ and reorganizing yields 
\begin{align*}
   \frac{1}{e_{k+1}}-\frac{1}{e_{k}} & \ge \frac{ce_k}{e_{k+1}}-\frac{d}{e_{k+1}e_k} \ge  
\begin{cases}
c - \frac{d}{2d/c}=\frac{1}{2}c,& d\neq 0\\
c,& d=0
\end{cases} .
\end{align*}
Summing, we obtain $\frac{1}{e_{k}}\ge \frac{1}{e_0}+ \frac{1}{2}kc$ when $d\neq 0$ and $\frac{1}{e_k} \geq \frac{1}{e_0}+ kc$ when $d=0$. Inverting both sides yields the claim.
\end{proof}

\begin{lemma}[Sequence analysis II]
\label{prop:seq4}
Consider a sequence $\{e_k\}_{k=0}^{\infty}$ with $e_{k} \geq 0$ and $e_{k+1} \leq (1-c)e_k + d$ for all $k$, where $c \in (0,2)$ and $d \geq 0$. Then $e_{k+1} \leq \left(1-c\right)^{k+1}e_0 + \frac{d}{c}$.
\end{lemma}
\begin{proof}
Applying the condition recursively, we get:
\begin{equation*}
e_{k+1} \leq (1-c)^{k+1}e_0 + \sum_{i=0}^{k+1}(1-c)^{i}d \leq (1-c)^{k+1}e_0 + \frac{d}{c}.
\end{equation*}
Thus proving the claim.
\end{proof}
We now prove the main result of this section:
\begin{proof}[Proof of Lemma~\ref{thm:Rel_Error_GD}]
From Lemma 2.1 in \citep{friedlander2012hybrid},
\begin{equation*}
f(x_{k+1}) \leq f(x_{k}) - \frac{1}{2L}\|\bfg_k\|_{2}^{2} + \frac{1}{2L}\|\hat{\bfg}_k - \bfg_k\|_{2}^{2},
\end{equation*}
and so:
\begin{align}
    & e_{k+1} \leq e_{k} - \frac{1}{2L}\|\bfg_k\|_{2}^{2} + \frac{\varepsilon_{\mathrm{rel}}}{2L}\|\bfg_k\|_{2}^{2} + \frac{\varepsilon_{\mathrm{abs}}}{2L} \nonumber \\
\Rightarrow~& e_{k+1} \leq e_{k} - \frac{1}{2L}\left(1 - \varepsilon_{\mathrm{rel}}\right)\|\bfg_k\|_{2}^{2} + \frac{\varepsilon_{\mathrm{abs}}}{2L}. \label{eq:Haul}
\end{align}
From convexity:
\begin{equation}
    e_{k} := f(x_{k}) - f^{\star} \leq \langle \bfg_{k}, x_{k} - \proj_{\star}(x_{k})\rangle~\Rightarrow~ \frac{e_{k}^{2}}{\|x_{k} - \proj_{\star}(x_{k})\|_{2}^{2}} \leq \|\bfg_{k}\|_{2}^{2}.
    \label{eq:Use_Convexity}
\end{equation}
Combining \eqref{eq:Haul} and \eqref{eq:Use_Convexity}:
\begin{equation}
e_{k+1} \leq e_{k} - \frac{1 - \varepsilon_{\mathrm{rel}}}{2L\|x_{k} - \proj_{\star}(x_{k})\|_{2}^{2}}e_{k}^{2} +  \frac{\varepsilon_{\mathrm{abs}}}{2L}.
\label{eq:appeal_to_this}
\end{equation}
To prove the first claim, use Proposition~\ref{prop:Boundedness_1} to get $\|x_{k} - \proj_{\star}(x_{k})\|_{2}^{2} \leq R^{2}$. Thus:
\begin{align*}
e_{k+1} \leq e_{k} - \frac{1 - \varepsilon_{\mathrm{rel}}}{2LR^2}e_{k}^{2} +  \frac{\varepsilon_{\mathrm{abs}}}{2L}   .
\end{align*}
Appealing to Lemma~\ref{prop:seq} completes the proof. For the second claim, by Assumption~\ref{def:res_convex} $e_{k} \geq \frac{\nu}{2} \|x_{k} - \proj_{\star}(x_{k})\|_{2}^{2}$. Substituting this into \eqref{eq:appeal_to_this}:
\begin{equation*}
e_{k+1} \leq e_{k} - \frac{(1 -\varepsilon_{\mathrm{rel}})\nu}{4L}e_k + \frac{\varepsilon_{\mathrm{abs}}}{2L}
\end{equation*}
and now using Lemma~\ref{prop:seq4} yields the claimed convergence rate. 
\end{proof}

\section{Prox-gradient descent using inexact gradients}
\label{sec:ProxGrad}
Here, we consider minimizing $F = f + r$ using prox-gradient descent (also known as forward-backward splitting): $x_{k+1} = \text{prox}_{\alpha r}\left(x_{k} -\alpha\hat{\bfg}_{k}\right)$. Recall the proximal operator is defined as:
\begin{equation*}
    \text{prox}_{\alpha r}(x) = \argmin_{y\in\mathbb{R}^d}  \frac{1}{2}\|x - y \|_{2}^{2} + \alpha r(y).
\end{equation*}
When exact gradients are available and $\text{prox}_{\alpha r}$ is exactly computable, prox-gradient descent is known to converge at the same rate as gradient descent \citep{ryu2020large}. This is particularly useful when $r$ is non-smooth. We consider the situation where one can compute $\text{prox}_{\alpha r}$ exactly but one only has access to inexact gradients of $f$ satisfying $\|\bfg_{k} - \hat{\bfg}_{k}\|_2^2 \leq \varepsilon_{\mathrm{abs}}$. In this section $e_{k} := F(x_k) - F^{\star}$.

\begin{theorem}
\label{thm:Convergence_Abs_Error_Reg}
Suppose $f$ is convex and satisfies Assumption~\ref{assumption:Lipschitz_Diff}. Suppose $r$ is convex and $F = f + r$ satisfies Assumptions~\ref{assumption:f_coercive} and \ref{assumption:subgrad_coercive}. Suppose $\|\bfg_{k} - \hat{\bfg}_{k}\|_2^2 \leq \varepsilon_{\mathrm{abs}}$ for all $k$ while $\textup{prox}_{\alpha r}(x)$ can be computed precisely. Choose $\alpha = \frac{1}{L}$. Then:
\begin{equation*}
e_{k} \leq \max\left\{\frac{36LR^{2}e_0}{e_0(k-t) + 36LR^2}, R\sqrt{10\varepsilon_{\mathrm{abs}}}\right\} \quad \text{ for } k \geq t,
\end{equation*}
where $\displaystyle t = \left\lceil\frac{\log(2e_0/3)}{\log(3/2)}\right\rceil$ and $R$ is a constant satisfying $\|x_k - \proj_{\star}(x_k)\|_2 \leq R$ for all $k$ determined by the coercivity conditions (\ref{assumption:f_coercive} and \ref{assumption:subgrad_coercive}). If instead of Assumption~\ref{assumption:Coercivity} $F$ satisfies Assumption~\ref{def:res_convex} then:
\begin{equation*}
e_{k} \leq \left(\frac{24L}{\nu + 24L}\right)^{k}e_0 + \frac{20\varepsilon_{\mathrm{abs}}}{\nu}.
\end{equation*}
\end{theorem}
In simpler terms, this theorem gives $\cO(1/k)$ convergence to an error horizon proportional to $R\sqrt{\varepsilon_{\mathrm{abs}}}$. If $F(x)$ is $\nu$-restricted strongly convex then we get linear convergence to an error horizon proportional to $\varepsilon_{\mathrm{abs}}/\nu$. \citep{schmidt2011convergence} proves a similar rate, without error horizon, for the case where $\|\bfg_{k} - \hat{\bfg}_{k}\|_2^2 \leq \varepsilon_{\mathrm{abs}}^{(k)}$ with the sequence $\varepsilon_{\mathrm{abs}}^{(k)}$ summable. 

\subsection{Deducing Theorem~\ref{thm:Main_regularized_sparse} from Theorem~\ref{thm:Convergence_Abs_Error_Reg}}
\label{section:DeducingMainTheorem}
Let us again explain how one can deduce the query complexity for ZORO (Theorem~\ref{thm:Main_regularized_sparse}) from Theorem~\ref{thm:Convergence_Abs_Error_Reg}. If $f$ satisfies Assumption~\ref{Assumption:ExactSparsity} then from Corollary~\ref{cor:ErrorBound3}:
\begin{equation*}
\|\hat{\bfg}_k - \bfg_k\|_2^2 \leq 9\tau^{2}\sigma H =: \varepsilon_{\mathrm{abs}}
\end{equation*}
for all $k$ with probability $1 - 2(s/d)^{b_2s}$, choosing $n$ sufficiently large so $\rho^{n}\|\bfg_{k}\|_2 \leq \tau\sqrt{\sigma H}$. For any $\varepsilon \geq R\sqrt{10\varepsilon_{\mathrm{abs}}} = R\sqrt{90\tau^{2}\sigma H} := b_5R\sqrt{\sigma H}$ we solve for $k$ to guarantee $e_{k} \leq \varepsilon$:
\begin{equation*}
k \geq \frac{36LR^{2}}{\varepsilon} + \log(e_0) \geq \frac{36LR^{2}}{\varepsilon} + t \geq  \frac{36LR^{2}}{\varepsilon} + t - \frac{36LR^{2}}{e_0}.
\end{equation*}
Multiplying this number by $b_1s\log(d)$, as in the proof of Theorem~\ref{thm:Main_compressible}, yields the first result. If $F$ is restricted strongly convex, then by the same line of reasoning for any $\varepsilon > 20\varepsilon_{\mathrm{abs}}/\nu = b_{5}^{2}\sigma H/\nu$, we get from Theorem~\ref{thm:Convergence_Abs_Error_Reg} that $e_{k} \leq \varepsilon$ for:
\begin{equation*}
k \geq \log\left(\frac{\varepsilon\nu - b_{5}^{2}\sigma H}{\nu e_0}\right)\bigg/\log\left(\frac{24L}{\nu + 24L}\right).
\end{equation*}
Multiplying by the number of queries per iteration yields the claimed query complexity. 

\subsection{Proof of Theorem~\ref{thm:Convergence_Abs_Error_Reg}}
\label{sec:ProofThm4.1}
Before proceeding we quantify, under very general conditions, the expected decrease per iteration. From the first order optimality condition of $\textup{prox}$: $x_{k+1} - x_{k} = -\alpha(\tnabla r(x_{k+1})+ \hat{\bfg}_{k})$,
where $\tnabla r(x_{k+1})\in\partial r(x_{k+1})$. It will be convenient to introduce the following notation:
\begin{align*}
    \text{Actual direction:}&&\hat{\Delta}_k &:= -(\tnabla r(x_{k+1})+ \hat{\bfg}_{k}),\\
    \text{Ideal direction:} && \Delta_k &:= -(\tnabla r(x_{k+1})+ {\bfg}_{k}),\\
    \text{Stationarity:} &&\tilde\Delta_k &:= -(\tnabla r(x_{k+1})+ {\bfg}_{k+1}).
\end{align*}

\begin{lemma}[Descent Lemma]
\label{lemma:boundB3}
Assume $f$ and $r$ are convex, $f$ satisfies Assumption~\ref{assumption:Lipschitz_Diff} and $\alpha = \frac{1}{L}$. Then:
\begin{align}
F(x_{k+1}) &\leq F(x_k)  -\frac{1}{4L} \|\hat{\Delta}_k\|_2^2 + \frac{1}{L}\|\bfg_k - \hat{\bfg}_k\|_2^2 \label{eq:bnd1},\\
F(x_{k+1}) &\leq F(x_k)  - \frac{1}{2L}\|\Delta_k\|_2^2 + \frac{1}{2L}\|\bfg_k - \hat{\bfg}_k\|_2^2 \label{eq:bnd1p},\\
F(x_{k+1}) &\leq F(x_k)  -\frac{1}{12L}\|\tilde{\Delta}_k\|_2^2 + \frac{5}{6L}\|\bfg_k - \hat{\bfg}_k\|_2^2. \label{eq:bnd1pp}
\end{align}
\end{lemma}

\begin{proof}
We begin by expanding $F$.
\begin{align} \label{eq:compsplit}
     F(x_{k+1}) - F(x_{k})  = f(x_{k+1})-f(x_k) + r(x_{k+1})-r(x_k).
\end{align}
By convexity of $r$ and the definition of $\hat{\Delta}_{k}$:
\begin{align*}
r(x_{k+1})-r(x_k) \le \langle \tnabla r(x_{k+1}), \alpha\hat{\Delta}_k\rangle  = -\langle \alpha\hat{\Delta}_k, \hat{\bfg}_{k}\rangle - \alpha\|\hat{\Delta}_k\|_2^2.
\end{align*}
Now for $f(x_{k+1})-f(x_k)$ in \eqref{eq:compsplit}, apply smoothness of $f(x)$ (Assumption~\ref{assumption:Lipschitz_Diff}) to get:
\begin{align*}
    f(x_{k+1})-f(x_k) &\le \langle \alpha\hat{\Delta}_k, {\bfg}_{k}\rangle + \frac{L}{2}\|\alpha\hat{\Delta}_k\|_2^2\\
    & = \langle \alpha\hat{\Delta}_k, \hat{\bfg}_{k}\rangle + \frac{L}{2}\|\alpha\hat{\Delta}_k\|_2^2 + \langle \alpha\hat{\Delta}_k,{\bfg}_{k}- \hat{\bfg}_{k}\rangle.
\end{align*}
Adding the bounds for $r(x_{k+1})-r(x_k)$ and $f(x_{k+1})-f(x_k)$ and using $\alpha = \frac{1}{L}$:
\begin{align*}
  F(x_{k+1}) - F(x_k)  &\leq - \frac{1}{2L} \|\hat{\Delta}_k\|_2^2  + \frac{1}{L}\langle\hat{\Delta}_k,{\bfg}_{k}- \hat{\bfg}_{k}\rangle \numberthis\label{eq:Fdesc0} \\
  &  \stackrel{(a)}{\le } - \frac{1}{2L} \|\hat{\Delta}\|_{2}^{2} + \frac{1}{4L}\|\hat{\Delta}\|_{2}^{2} + \frac{1}{L}\|\bfg_k - \hat{\bfg}_k\|_{2}^{2} \\
  &= -\frac{1}{4L}\|\hat{\Delta}\|_{2}^{2} + \frac{1}{L} \|\bfg_k - \hat{\bfg}_k\|_{2}^{2},
\end{align*}
where we have used Young's inequality to obtain (a). For the second inequality, we return to \eqref{eq:Fdesc0} and apply $\hat{\Delta}_k = \Delta_k + (\bfg_k - \hat{\bfg}_{k})$. This yields:
\begin{align*}
F(x_{k+1}) - F(x_k) & \leq -\frac{1}{2L}\left(\|\Delta_k\|_2^2 +2\langle \Delta_k,\bfg_k - \hat{\bfg}_k\rangle + \|\bfg_k - \hat{\bfg}_k\|_{2}^{2}\right)\\
&~\quad + \frac{1}{L}\left(\langle \Delta_k,\bfg_k - \hat{\bfg}_k\rangle + \|\bfg_k - \hat{\bfg}_k\|_{2}^{2}\right) \\
        & \leq -\frac{1}{2L}\|\Delta_k\|_2^2 + \frac{1}{2L}\|\bfg_k - \hat{\bfg}_k\|_{2}^{2} .
\end{align*}

Finally, observe that:
\begin{align*}
    \|\tilde{\Delta}_k\|_2 & \stackrel{(a)}{=} \|-\tnabla r(x_{k+1}) - \bfg_{k+1}\|_2\\
    &= \|-\tnabla r(x_{k+1}) - \bfg_{k} + \bfg_k - \bfg_{k+1}\|_2 \\
    & \stackrel{(b)}{\leq} \|\Delta_k\|_2 + \|\bfg_k - \bfg_{k+1}\|_2 \stackrel{(c)}{\leq} \|\Delta_k\|_2 + \alpha L\|\hat{\Delta}_k\|_2,
\end{align*}
where (a) follows from the definition of $\tilde{\Delta}_k$, (b) follows from the definition of $\Delta_k$ and (c) follows from smoothness of $f(x)$ (Assumption~\ref{assumption:Lipschitz_Diff}). Using $\alpha = \frac{1}{L}$ we have $\|\tilde{\Delta}_k\|_2^2 \leq 2\|\Delta_k\|_2^2 + 2\|\hat{\Delta}_k\|_2^2$. Therefore, combining \eqref{eq:bnd1} and \eqref{eq:bnd1p} yields \eqref{eq:bnd1pp}.
\end{proof}

An immediate consequence of Lemma~\ref{lemma:boundB3} is the following result:
\begin{theorem}
\label{thm:Convergence_to_stationarity}
Suppose $f$ satisfies Assumption~\ref{assumption:Lipschitz_Diff}, $\|\bfg_k - \hat{\bfg}_k\|_2^2 \leq \varepsilon_{\mathrm{abs}}$ for all $k$ and $\mathrm{prox}_{r}(x)$ can be computed precisely. If $\alpha = \frac{1}{L}$ then:
\begin{equation*}
\min_{\ell=0,\ldots,k-1}\|\tilde{\Delta}_{\ell}\|_{2} \leq \frac{\sqrt{12Le_0}}{\sqrt{k}} + \sqrt{10\varepsilon_{\mathrm{abs}}},
\end{equation*}
\end{theorem}
\begin{proof}
Summing equation \eqref{eq:bnd1pp}:
\begin{align*}
& \frac{1}{12L} \sum_{\ell=0}^{k-1}\|\tilde{\Delta}_{\ell}\|_{2}^{2} \leq F(x_0) - F(x_{k}) + \frac{5}{6L} \sum_{\ell=0}^{k-1} \|\bfg_k - \hat{\bfg}_k\|_{2}^{2} \leq e_{0} + \frac{5}{6L}\left(k\varepsilon_{\mathrm{abs}}\right) \\
\Rightarrow~& \frac{1}{12L}k \min_{\ell=0,\ldots,k-1}\|\tilde{\Delta}_{\ell}\|_{2}^{2} \leq e_{0} + \frac{5}{6L}\left(k\varepsilon_{\mathrm{abs}}\right) \\
\Rightarrow~& \min_{\ell=0,\ldots,k-1}\|\tilde{\Delta}_{\ell}\|_{2} \leq \sqrt{\frac{12Le_0}{k} + 10\varepsilon_{\mathrm{abs}}} \leq \frac{\sqrt{12Le_0}}{\sqrt{k}} + \sqrt{10\varepsilon_{\mathrm{abs}}} .
\end{align*}
\end{proof}

Theorem~\ref{thm:MainNonConvex} follows easily from Theorem~\ref{thm:Convergence_to_stationarity} by using Corollary~\ref{cor:ErrorBound3} and $\sqrt{10\varepsilon_{\mathrm{abs}}} = b_5\sqrt{\sigma H}$ (as in Section~\ref{section:DeducingMainTheorem}).

\begin{lemma}[Sequence analysis III]
\label{prop:seq1}
Consider a sequence $\{e_k\}_{k=0}^{\infty}$ with $e_k \geq 0$ and 
\begin{equation}
e_{k+1}+ c e_{k+1}^2\le e_k +d
\label{eq:Main_Recurrence_3}
\end{equation}
where $c>0,d\ge0$. Let $t=\lceil \log(2ce_0/3)/\log(3/2)\rceil$. Then $e_k \le \frac{3e_0}{e_0c(k-t) +3}$ for $k\in \{t:e_{0},\dots,e_{t+1} \ge \sqrt{3d/c}\}\cap \{k: k \geq t\}$.
\end{lemma}
Note that the logarithm makes $t$ much smaller than $ce_0$.
\begin{proof}
Partition $\{0,1,\dots\}$ into disjoint sets $\mathcal{A} := \{k: e_{k+1} < 2 e_k/3\}$ and $\mathcal{B} := \{k: e_{k+1} \ge 2 e_k/3\}$. For $k\in\mathcal{A}$ we have geometric decrease:
$1/e_{k+1} - 1/e_k > 1/(2e_k)$. 

In the rest of this proof, we restrict ourselves to the set of $k$ such that $e_{k+1} \ge \sqrt{3d/c}$ and no longer state it explicitly. For example, by $k\in\mathcal{A}$, we mean $k$ in the \emph{intersection}  of $\mathcal{A}$ and the restriction. Note that, when $d=0$, the set has all $k\ge 0$. The restriction gives us
\begin{equation*}
e_{k+1}\le e_k + d - c e_{k+1}^2\le e_k - 3d,
\end{equation*}
so $\{e_k\}$ is monotonically non-increasing. It takes at most $t$ first entries $k_1,\dots,k_t\in\mathcal{A}$ to ensure, for $k\ge k_t$, we have $1/(2e_k)\ge c/3$, so define $\mathcal{A}'=\mathcal{A}\setminus\{k_1,\dots,k_t\}$ and we get $1/e_{k+1} - 1/e_k > c/3$ for $k\in\mathcal{A}'$. Dividing \eqref{eq:Main_Recurrence_3} by $e_{k+1}e_k$ and reorganizing:
\begin{align*}
   \frac{1}{e_{k+1}}-\frac{1}{e_{k}} &\ge   
\begin{cases}
\frac{c}{3},& k\in\mathcal{A}'\\
\frac{c e_{k+1}}{e_{k}}-\frac{d}{e_{k+1}e_k} \ge \frac{2c}{3} - \frac{d}{3d/c}=\frac{c}{3},& d\neq 0, k\in\mathcal{B}\\
\frac{c e_{k+1}}{e_{k}}-\frac{d}{e_{k+1}e_k} \ge  \frac{2c}{3},& d=0, k\in\mathcal{B}
\end{cases}\\
 &\ge \frac{c}{3},\qquad \qquad k\in \mathcal{A}'\cup\mathcal{B}.
\end{align*}
So except for no more than $t$ iterations, we have
$\frac{1}{e_k}\ge \frac{1}{e_0}+\frac{kc}{3}$. In each of those $t$ iterations, $e_k$ does not increase. Hence, the stated result holds.
\end{proof}

\begin{lemma}[Sequence analysis IV]
\label{prop:seq2}
Consider a sequence $\{e_k\}_{k=0}^{\infty}$ with $e_k \geq 0$ and $(1+a)e_{k+1} \leq e_{k} + b$, where $a,b > 0$, for all $k$. Then $e_{k+1} \leq \left(\frac1{1+a}\right)^{k+1} e_{0} + \frac{b}{a}$ for $k > 0$.
\end{lemma}

\begin{proof}
Rearranging the recurrence relation gives us $\displaystyle e_{k+1} \leq \frac{1}{1+a} e_{k} + \frac{b}{1+a}$. Applying this recursively yields:
\begin{align*}
e_{k+1} &\leq \left(\frac{1}{1+a}\right)^{k+1} e_{0} + \frac{b}{1+a} \sum_{\ell=0}^{k}\left(\frac{1}{1+a}\right)^{\ell} \\
&\leq \left(\frac{1}{1+a}\right)^{k+1} e_{0} + \frac{b}{1+a}\left(\frac{1+a}{a}\right) = \left(\frac{1}{1+a}\right)^{k+1} e_{0} + \frac{b}{a}.
\end{align*}
\end{proof}

 We now prove our main results by showing that $e_k$ satisfies the recurrences described in Propositions~\ref{prop:seq1} and \ref{prop:seq2}. 

\begin{lemma}\label{lm:eineq}
Under the assumptions of Theorem~\ref{thm:Convergence_Abs_Error_Reg}:
\begin{align*}
    e_{k+1}+ \frac{1}{12L\|x_{k+1} - \proj_{\star}(x_{k+1})\|_2^2} e_{k+1}^2 \le e_k + \frac{5}{6L}\|{\bfg}_{k}- \hat{\bfg}_{k}\|_2^2.
\end{align*}
\end{lemma}

\begin{proof}
Recall $\tilde{\Delta}_{k} := \tnabla r(x_{k+1}) + \bfg_{k+1}$. By convexity of $F$ and the Cauchy-Schwarz inequality:
\begin{align*}
& e_{k+1} = F(x_{k+1}) - F^{\star} \le \langle \tilde{\Delta}_{k}, x_{k+1}-\proj_{\star}(x_{k+1})\rangle \le\|\tilde{\Delta}_{k}\|_2\|x_{k+1} - \proj_{\star}(x_{k+1})\|_2 \\
\Rightarrow~ &  \frac{e_{k+1}^{2}}{\|x_{k+1} - \proj_{\star}(x_{k+1})\|_2^2} \leq \|\tilde{\Delta}_{k}\|_{2}^{2} .
\end{align*}
Rearranging \eqref{eq:bnd1pp}, we get $\frac{1}{12L} \|\tilde{\Delta}_k\|_{2}^{2} \leq e_{k} - e_{k+1} + \frac{5}{6L}\|\bfg_k - \hat{\bfg}_k\|_{2}^{2}$ and hence:
\begin{equation*}
\frac{1}{12L\|x_{k+1} - \proj_{\star}(x_{k+1})\|_2^2}e_{k+1}^{2} \leq e_{k} - e_{k+1} + \frac{5}{6L}\|\bfg_k - \hat{\bfg}_k\|_{2}^{2}.
\end{equation*}
\end{proof}

Using boundedness of the iterates (proved below in Section~\ref{sec:Boundedness}) and the above lemmas we now prove the main result of this section.

\begin{proof}[Proof of Theorem~\ref{thm:Convergence_Abs_Error_Reg}]
Appealing to Lemma~\ref{lm:eineq} and using $\|\bfg_k - \hat{\bfg}_k\|_2^2 \leq \varepsilon_{\mathrm{abs}}$:
\begin{align*}
    e_{k+1}+ \frac{1}{12L\|x_{k+1} - \proj_{\star}(x_{k+1})\|_2^2} e_{k+1}^2 \le e_k + \frac{5\varepsilon_{\mathrm{abs}}}{6L}.
\end{align*}
By Proposition~\ref{prop:Boundedness_2} there exists an $R > 0$ such that $\|x_{k+1} - \proj_{\star}(x_{k+1})\|_2\leq R$ for all $k$. Apply Lemma~\ref{prop:seq1} with
$c :=  \frac{1}{12LR^{2}}$ and $d := \frac{5\varepsilon_{\mathrm{abs}}}{6L}$ to obtain part 1 of Theorem~\ref{thm:Convergence_Abs_Error_Reg}. If $F$ is restricted strongly convex then from Lemma~\ref{lm:eineq} and the definition of restricted strong convexity:
\begin{equation*}
e_{k+1} + \frac{\nu}{24L}e_{k+1} \leq e_{k} + \frac{5\varepsilon_{\mathrm{abs}}}{6L}.
\end{equation*}
Now apply Lemma~\ref{prop:seq2} with $a = \frac{\nu}{24L}$ and $b = \frac{5\varepsilon_{\mathrm{abs}}}{6L}$ to obtain:
\begin{equation*}
e_{k+1} \leq \left(\frac{24L}{\nu + 24L}\right)^{k+1}e_0 + \frac{20\varepsilon_{\mathrm{abs}}}{\nu}.
\end{equation*}
\end{proof}

\section{Boundedness}
\label{sec:Boundedness}
In this section, we show the coercivity assumptions (Assumption~\ref{assumption:Coercivity}) are sufficient to guarantee the sequence $\{\|x_k - \proj_{\star}(x_k)\|_2\}_{k=1}^{\infty}$ is bounded. 

\begin{proposition}
\label{prop:Boundedness_1}
Suppose $f$ is convex and satisfies Assumptions~\ref{assumption:Lipschitz_Diff} and \ref{assumption:f_coercive} while $\nabla f$ satisfies Assumption~\ref{assumption:grad_coercive}. Suppose $\|\bfg_k - \hat{\bfg}_k\|_2^2 \leq \varepsilon_{\mathrm{abs}} + \varepsilon_{\mathrm{rel}}\|\bfg_k\|_2^2$ for all $k$ with $\varepsilon_{\mathrm{rel}} < 1$. Let $\alpha = 1/L$. Then there exists an $R > 0$ such that $\|x_{k} - \proj_{\star}(x_k)\|_{2} \leq R$ for all $k$.
\end{proposition}

\begin{proof}
This proof proceeds via three steps.
\begin{enumerate}
    \item For any $\beta > f^{\star}$,  define the level set $\mathcal{L}_{\beta}:= \{x: \ f(x) \leq \beta\}$. As $f$ is coercive, for any $\beta$, there exists an $R_{\beta}> 0$ such that if $\|x - \proj_{\star}(x)\|_2 > R_{\beta}$ then $f(x) > \beta$. Equivalently, $\|x - \proj_{\star}(x)\|_2 \leq R_{\beta}$ for all $x\in \mathcal{L}_{\beta}$.

    \item Rewriting \eqref{eq:Haul}, we obtain:
\begin{equation}
f(x_{k+1}) - f(x_{k}) \leq - \underbrace{\frac{1}{2L}\left(1 - \varepsilon_{\mathrm{rel}}\right)}_{=a}\|\bfg_k\|_{2}^{2} + \underbrace{\frac{\varepsilon_{\mathrm{abs}}}{2L}}_{=b}.
\label{eq:Bounded_Contradiction}
\end{equation}
As $\varepsilon_{\mathrm{rel}}< 1$,  we have $a > 0$. As $\nabla f$ is coercive with respect to $f$, there exists a $Q$ such that if $f(x) \geq Q$ then $\|\bfg_k\|_2 := \|\nabla f(x)\|_2 \geq \sqrt{b/a}$ (where $a,b$ are as in \eqref{eq:Bounded_Contradiction}).  We use this to establish, via induction, that:
\begin{equation}
    f(x_k) \leq \max\{f(x_0), Q\} + b.
    \label{eq:Induction}
\end{equation}
From \eqref{eq:Bounded_Contradiction} one easily checks the base case: $f(x_1) \leq f(x_0) + b$. Suppose that \eqref{eq:Induction} holds at the $k$-th step. Then either $f(x_k) \leq Q$, in which case appealing to \eqref{eq:Bounded_Contradiction}, we obtain:
\begin{equation*}
  f(x_{k+1}) \leq f(x_k) + b \leq Q + b \leq \max\{f(x_0), Q \} + b, 
\end{equation*}
or $f(x_k) > Q$ whence (by the coercivity of $\nabla f$) we obtain:
\begin{align*}
    & f(x_{k+1}) - f(x_k) \leq -a\left(\frac{b}{a}\right) + b = 0 \\
    \Rightarrow & f(x_{k+1}) \leq f(x_k) \leq \max\{f(x_0), Q\} + b. \quad \text{\em (By induction hypothesis)}
\end{align*}
    \item Finally, let $\beta = \max\{f(x_0), Q\} + b$. From part 2, it follows that $x_k \in \mathcal{L}_{\beta}$ for all $k$. From part 1, it then follows that there exists an $R := R_{\beta}$ such that $\|x_k - \proj_{\star}(x_k)\|_2 \leq R$ for all $k$, thus proving the theorem.
\end{enumerate}
\end{proof}  
Our second result allows for regularization ($r \neq 0$) but requires $\varepsilon_{\mathrm{rel}} = 0$.

\begin{proposition}
\label{prop:Boundedness_2}
Suppose that $f$ satisfies Assumption~\ref{assumption:Lipschitz_Diff}, $F$ satisfies Assumption~\ref{assumption:f_coercive} and $\partial F$ satisfies Assumption~\ref{assumption:subgrad_coercive}. Suppose $\|\bfg_k - \hat{\bfg}_k\|_2^2 \leq \varepsilon_{\mathrm{abs}}$ for all $k$. Let $\alpha = 1/L$. Then there exists an $R > 0$ such that $\|x_k - \proj_{\star}(x_k)\|_{2} \leq R$ for all $k$.
\end{proposition}

\begin{proof}
This proof is similar to that of Proposition~\ref{prop:Boundedness_1}.
\begin{enumerate}
    \item For any $\beta \geq F^{\star}$ define the level set $\mathcal{L}_{\beta}:= \{x: \ F(x) \leq \beta\}$. Again, as $F$ is coercive there exists an $R_{\beta} > 0$ such that $\|x - \proj_{\star}(x)\|_2 \leq R_{\beta}$ for all $x \in \mathcal{L}_{\beta}$. 
    \item Using $\|\bfg_k - \hat{\bfg}_k\|_2^2 \leq \varepsilon_{\mathrm{abs}}$, from \eqref{eq:bnd1pp}, we get:
\begin{equation}
F(x_{k+1}) - F(x_{k}) \leq -\underbrace{\frac{1}{12L}}_{=a}\|\tilde{\Delta}_k\|_{2}^{2} + \underbrace{\frac{5\varepsilon_{\mathrm{abs}}}{6L}}_{=b}.
\label{eq:Bounded_Contradiction_2}
\end{equation}
Because $\partial F$ is coercive with respect to $F$, there exists a $Q$ such that if $F(x) \geq Q$ then $\|u\|_2 \geq \sqrt{b/a}$ for all $u \in \partial F(x)$. We now establish, via induction, that:
\begin{equation}
    F(x_{k+1}) \leq \max\{F(x_0) +b, Q\}.
    \label{eq:Induction2}
\end{equation}
The base case ($k=0$) is easily verified. So, suppose \eqref{eq:Induction2} holds at $k-1$, {\em i.e.} $F(x_k) \leq \max\{F(x_0)+b, Q\}$. Now, either $F(x_{k+1}) < Q$ or $F(x_{k+1}) \geq Q$. If the former holds, we are done. So, suppose the latter. By the coercivity of $\partial F$, we have $\|\tilde{\Delta}_k\|_2^2 \geq b/a$ as $\tilde{\Delta}_k\in \partial F(x_{k+1})$ (see the definition in Section~\ref{sec:ProofThm4.1}) whence by \eqref{eq:Bounded_Contradiction_2},
\begin{equation}
    F(x_{k+1}) \leq F(x_k) \leq \max\{F(x_0)+b, Q\}. \quad \text{\em (by induction hypothesis)}
\end{equation}

    \item Finally, let $\beta = \max\{F(x_0)+b, Q\}$. From part 2, it follows $x_{k+1} \in \mathcal{L}_{\beta}$ for all $k \geq 0$, and $x_0 \in \mathcal{L}_{\beta}$ by construction. Appealing to part 1, there again exists an $R := R_{\beta}$ such that $\|x_k - \proj_{\star}(x_k)\|_2 \leq R$ for all $k$.
\end{enumerate}
\end{proof}

For any adversarially noisy oracle (Assumption~\ref{assumption:noise model}), these coercivity conditions are also necessary. To see this, consider the following one-dimensional example. Take $r=0$ and let $f$ be the Huber loss function:
\begin{equation*}
f(x) = \left\{\begin{array}{ll} \frac{1}{2}x^{2}, & \text{for } |x| \leq m\\
m(|x| - \frac{1}{2}m), & \text{otherwise} \end{array}\right. .
\end{equation*}
While $f$ is coercive, $\nabla f$ is not coercive with respect to $f$ ({\em i.e.} $f$ does not satisfy Assumption~\ref{assumption:grad_coercive}). Suppose $\sigma$ in Assumption~\ref{assumption:noise model} satisfies $\sigma > m^{2}$. From Corollary~\ref{cor:ErrorBound3}, we get, at worst,
\begin{equation*}
\|\bfg_k - \hat{\bfg}_k\|_2 \approx 2\tau m > 2m \geq 2\|\bfg_k\|_2 \quad \text{(as $H = 1$).}
\end{equation*}
That is, for all $k$, the noise can be chosen adversarially such that $\text{sign}(\hat{\bfg}_k) \neq \text{sign}(\bfg_k)$, hence the inexact gradient descent may diverge.

\section{Adaptive sampling}
\label{sec:AdaZORO}

\begin{algorithm}[tb] 
   \caption{Adaptive ZORO (AdaZORO)} \label{algo:Adazoro}
\begin{algorithmic}[1]
   \State {\bfseries Input:} $x_0$: initial point; $s$: initial gradient sparsity level; $\alpha$: step size; $\delta$: query radius, $K$: number of iterations; $\phi$: error tolerance for adaptive sampling.
   \State $m \gets b_1 s\log(d/s)$ \quad {where $b_1$ is as in Theorem~\ref{thm:SatisfiesRIP}. Typically, $b_1 \approx 1$ is appropriate}. 
   \State $z_1,\dots,z_m\gets$ i.i.d. Rademacher random vectors
   \For{$k=0$ {\bfseries to} $K$}
        \For{$i=1$ {\bfseries to} $s$} \label{line:least_square_begin}
            \State $y_{i} \gets (E_f(x+\delta z_i)-E_f(x))/\delta$
        \EndFor
        \If{$k> 0$} 
        \State $\bfy\gets \frac{1}{\sqrt{s}}[y_1,\ldots, y_s]^{\top}$
        \State $Z\gets\frac{1}{\sqrt{s}}[z_1,\ldots, z_s]^{\top}$
        \State $ {\hat{\bfg}}_k\gets \argmin_{\bfg} \|Z\bfg - \bfy\|_2 \quad \text{s.t. } \supp(\bfg) = \supp({\hat{\bfg}}_{k-1})$
        \If{$\|Z{\hat{\bfg}}_k - \bfy\|_{2}/\|\bfy\|_2 \leq \phi$}
            \State  Goto Line~\ref{line:gd} 
        \EndIf
        \EndIf \label{line:least_square_end}
        \For{$i=s+1$ {\bfseries to} $m$} \label{line:same_s_start}
            \State $y_{i} \gets (E_f(x+\delta z_i)-E_f(x))/\delta$
        \EndFor
        \State $\bfy\gets \frac{1}{\sqrt{m}}[y_1,\ldots, y_m]^{\top}$
        \State $Z\gets\frac{1}{\sqrt{m}}[z_1,\ldots, z_m]^{\top}$
        \State ${\hat{\bfg}}_k \approx \argmin_{\|\mathbf{g}\|_{0} \leq s}\|Z\mathbf{g} - \mathbf{y}\|_2 \quad $ by CoSaMP \label{line:same_s_end}
        \While{$\|Z{\hat{\bfg}}_k - \bfy\|_2/\|\bfy\|_2 > \phi$} \label{line:check_accurate} \label{line:adapt_sample_begin}
            \State $s\gets s+1$
            \State $m^{\textrm{new}} \gets b_1 s\log(d/s)$
            \State Generate additional Rademacher random vectors $z_{m+1},\ldots, z_{m^{\textrm{new}}}$
            \For{$i=m+1$ {\bfseries to} $m^{\textrm{new}}$}
                \State $y_{i} \gets (E_f(x+\delta z_i)-E_f(x))/\delta$
            \EndFor
            \State $m\gets m^{\textrm{new}}$
            \State $\bfy\gets \frac{1}{\sqrt{m}}[y_1,\ldots, y_m]^{\top}$
            \State $Z\gets\frac{1}{\sqrt{m}}[z_1,\ldots, z_m]^{\top}$
            \State ${\hat{\bfg}}_k \approx \argmin_{\|\mathbf{g}\|_{0} \leq s}\|Z\mathbf{g} - \mathbf{y}\|_2 \quad $ by CoSaMP
        \EndWhile  \label{line:adapt_sample_end}
       \State \label{line:gd} $x_{k+1}\gets\mathbf{prox}_{\alpha r}(x_k-\alpha \hat{\bfg}_k)$
   \EndFor
   \State {\bfseries Output:} $x_K$: minimizer of \eqref{eq:reg_opt_problem}.
\end{algorithmic}
\end{algorithm}

Ideally, ZORO should work without assuming gradient compressibility, but still be able to exploit this when it arises. Moreover, the support of the previous gradient estimate is important information that should not be ignored. We present an algorithm, coined Adaptive ZORO (AdaZORO), incorporating these observations as Algorithm~\ref{algo:Adazoro}. Informally, the gradient estimator in AdaZORO proceeds as:
\begin{enumerate}
    \item Take only $s$ queries and solve a least squares problem with support restricted to $\supp(\bfg_{k-1})$. If the relative error in this solution is small, return it immediately as
    $\hat{\bfg}_k$ (Lines~\ref{line:least_square_begin}--\ref{line:least_square_end}); otherwise, move to next step. This saves queries when $\supp(\bfg(x_k))$ changes slowly.
    \item Take an additional $m - s$ queries and run the compressible gradient estimator (Lines~\ref{line:same_s_start}--\ref{line:same_s_end}), reusing the $s$ samples we have already taken. Check whether the solution is sufficiently accurate (Line~\ref{line:check_accurate}).
    \item Until a sufficiently accurate solution is found, increase $s$ and take additional oracle queries while retaining the earlier samples. Continue to estimate the gradient from these samples, old and new 
    (Lines~\ref{line:adapt_sample_begin}--\ref{line:adapt_sample_end}).
\end{enumerate}
At worst, AdaZORO makes $\cO(d)$ queries per iteration. Empirically, AdaZORO works very well on real-world datasets. For example, in the asset risk management experiment (see Section~\ref{subsec:asset risk}), $\bfg(x)\in\mathbb{R}^{225}$ but AdaZORO converges rapidly with approximate gradients having only $40$ to $72$ nonzero entries.

\section{Numerical experiments} \label{sec:numberical}

We compared ZORO to FDSA \citep{Kiefer1952}, SPSA \citep{spall1998overview} and the LASSO-based algorithm from \citep{Wang2018} (abbreviated as LASSO in the rest of this section). We did not test any global algorithms ({\em e.g.} REMBO \citep{wang2013bayesian}) due to their strong correlations with problem structures. 
For the synthetic experiments in Sections~\ref{sec:Synth_Datasets} and \ref{sec:numerical_speed_comparsion}, we only use vanilla ZORO ({\em i.e.} Algorithm~\ref{algo:zoro}), so we do not gain extra advantage with adaptive sampling.
We use AdaZORO ({\em i.e.} Algorithm~\ref{algo:Adazoro}) for the real-world asset risk management problem in Section~\ref{subsec:asset risk}. For sparse adversarial attack problem in Section~\ref{subsec:sparse attack}, we find the gradients are highly compressible and using vanilla ZORO is sufficient. A sample implementation of ZORO and AdaZORO can be found online at \url{https://github.com/caesarcai/ZORO}.

\subsection{Synthetic dataset I: Query efficiency} \label{sec:Synth_Datasets}
Consider $f(x) = x^{\top} A x/2$, where $A\in \mathbb{R}^{200\times 200}$ is a diagonal matrix. We tested: (a) exactly sparse case with 20 randomly generated diagonal positive numbers; (b) compressible case where diagonal elements are non-zeros that diminish exponentially---$A_{i,i} = e^{-\omega i}$ with $\omega=0.5$. For both cases, we used two versions of ZORO: ZORO(CoSaMP) without enforcing any constraints and ZORO(CoSaMP$+$prox) with a proximal operator to enforce non-negativity.

\begin{figure}
\centering
\hfill
\subfloat[Exact sparse case.]{
\includegraphics[scale = 0.4]{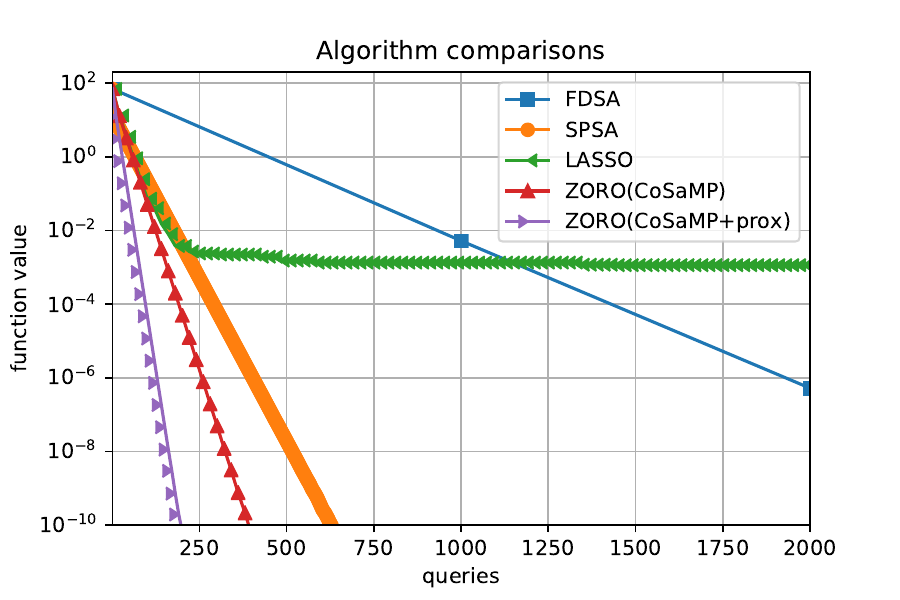} \label{fig_quad}}
\hfill
\subfloat[Compressible case.]{
\includegraphics[scale = 0.4]{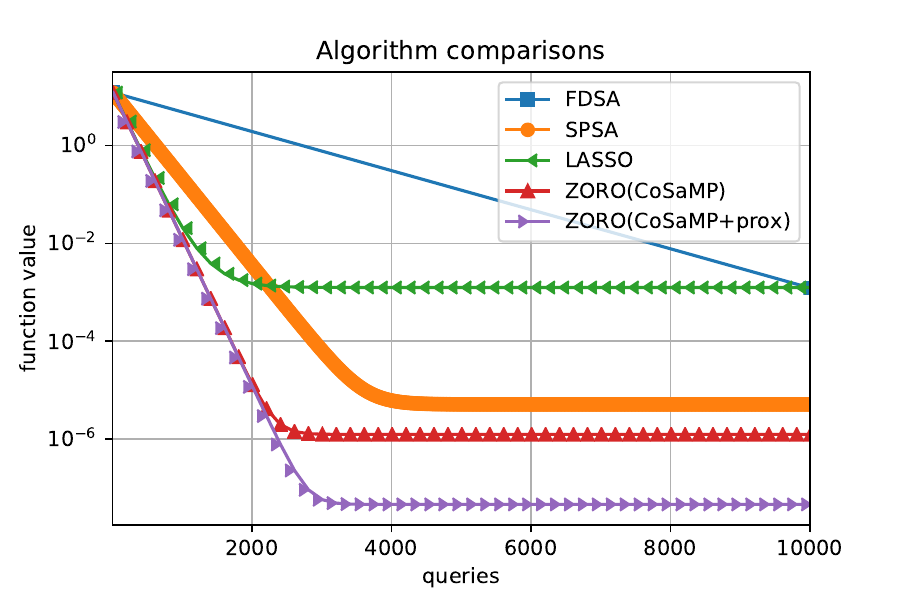} \label{fig_quad_dense}}
\hfill
\vspace{-0.05in}
\caption{Function values  \textit{v.s.} queries for gradient estimation methods in synthetic examples.  } 
\label{fig:Algo_Comparisons}
\end{figure}

The results of case (a) are shown in Figure~\ref{fig_quad}. ZORO(CoSaMP$+$prox) required $1/10$-th of the queries that FDSA required, and a third of the queries required by SPSA. Without the proximal operator, ZORO(CoSaMP) has less of an advantage, but is still noticeably cheaper than FDSA and SPSA, in terms of queries. LASSO consistently gets stuck around an accuracy of $10^{-3}$, and required more queries than both versions of ZORO before it converged.

The results of test (b) are summarized in Figure~\ref{fig_quad_dense}. Due to the ill-conditioned nature of the problem, none of the tested methods converged to arbitrarily small accuracy. However, ZORO(CoSaMP$+$prox) and ZORO(CoSaMP) achieved the best and second-best accuracy, respectively. They also exhibited the best query efficiency.

We conducted two additional experiments to further illustrate the advantages of using ZORO over SPSA\footnote{Note that our implementation of SPSA uses Rademacher random perturbation vectors $z$. Thus, it coincides with Random Search.}. We used the following two challenging objective functions:

\begin{figure}
\centering
\hfill
\subfloat[Max-$s$-squared-sum function.]{
\includegraphics[scale = 0.4]{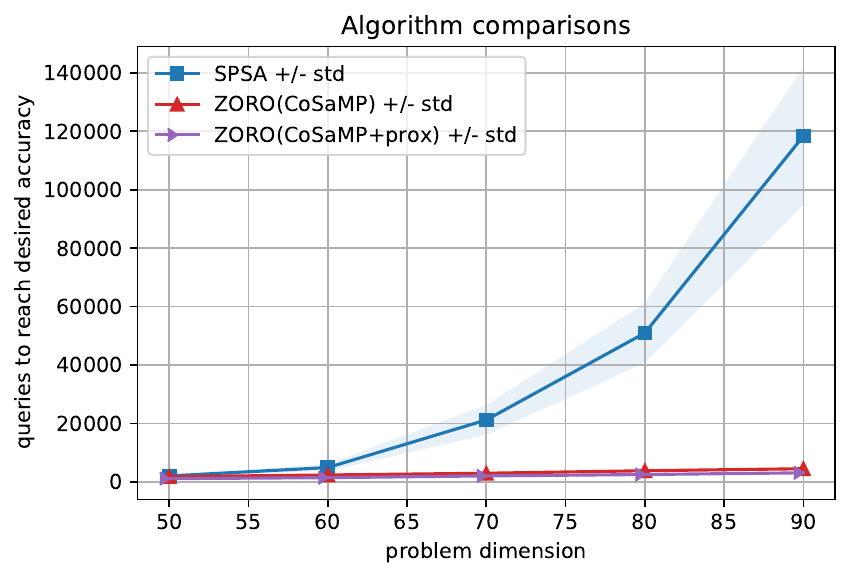} \label{fig_maxk}}
\hfill
\subfloat[Rotated quadratic function.]{
\includegraphics[scale = 0.4]{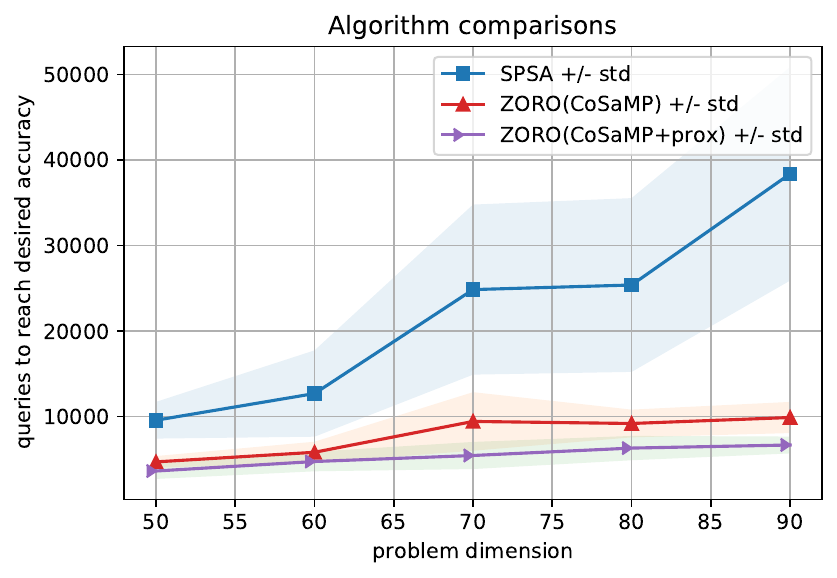} \label{fig_rotate_q}}
\hfill
\vspace{-0.05in}
\caption{Queries to reach desired accuracy  \textit{v.s.} problem dimension for gradient estimation methods in synthetic examples. } 
\label{fig:query_vs_dimension}
\end{figure}

\paragraph{Max-$s$-squared-sum function} $f(x)=\sum_{i=1}^{20} x_{m_i}^2$, where $x_{m_i}$ is the $i$-th largest-in-magnitude entry of $x$. This function has sparse gradients, and $\bfg(x)$ achieves {\em every possible support set} $S\subset [d]$.

\paragraph{Rotated sparse quadratic function} Pick an arbitrary sparse binary vector $x_\mathrm{true}$ with 10\% randomly located entries equal to 1. Consider the function $f(x)=(x-x_\mathrm{true})^{\top} QDQ^{\top} (x - x_\mathrm{true})$, where $Q$ is a random orthonormal matrix and $D$ is a diagonal matrix with uniform $[0,1]$ random entries. 
The gradients of $f(x)$ are not obviously compressible, but empirically most randomly sampled gradients are. The solution $x_\mathrm{true}$ is sparse, so we can incorporate this prior knowledge using a regularizer to accelerate convergence.

For both functions we repeated 10 experiments per dimension, using random initial points with unit $\ell_2$ norm. We ran each experiment until the objective error reached a threshold of $0.1\%$ of the initial objective error. Figure~\ref{fig:query_vs_dimension} depicts the means and standard deviations of the number of queries used for both functions. As is apparent, the query complexity of ZORO increases much more slowly than that of SPSA. In Figure~\ref{fig_maxk}, where $s=20$ is fixed while $d$ increases, it is clear that the query complexity of ZORO is only weakly dependent on $d$.

\subsection{Synthetic dataset II: Computational efficiency} \label{sec:numerical_speed_comparsion}

\begin{figure}[t]
\centering
\hfill
\subfloat[Runtime {\em v.s.} problem dimension.]{
\includegraphics[scale = 0.4]{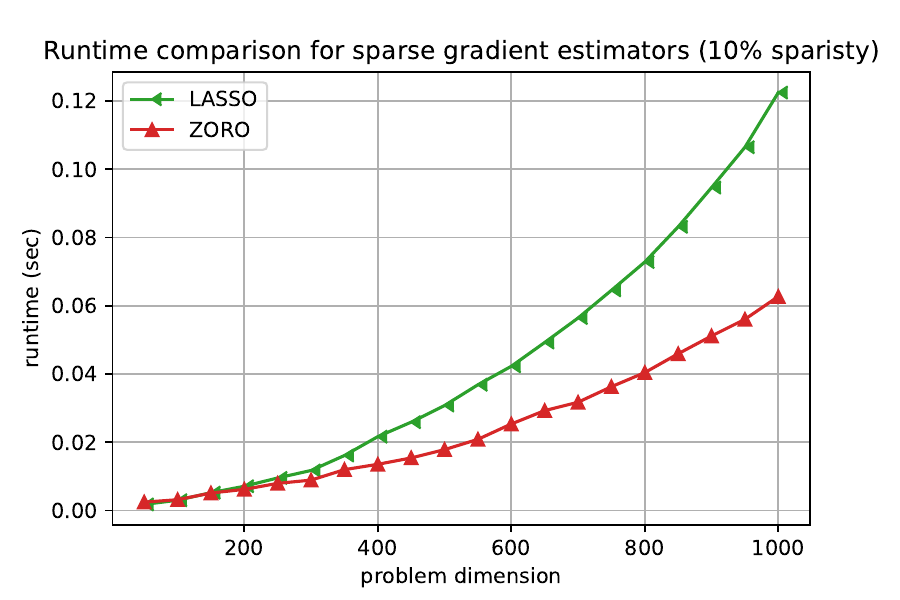} \label{fig_time1}}
\hfill
\subfloat[Runtime {\em v.s.} sparsity level.]{
\includegraphics[scale = 0.4]{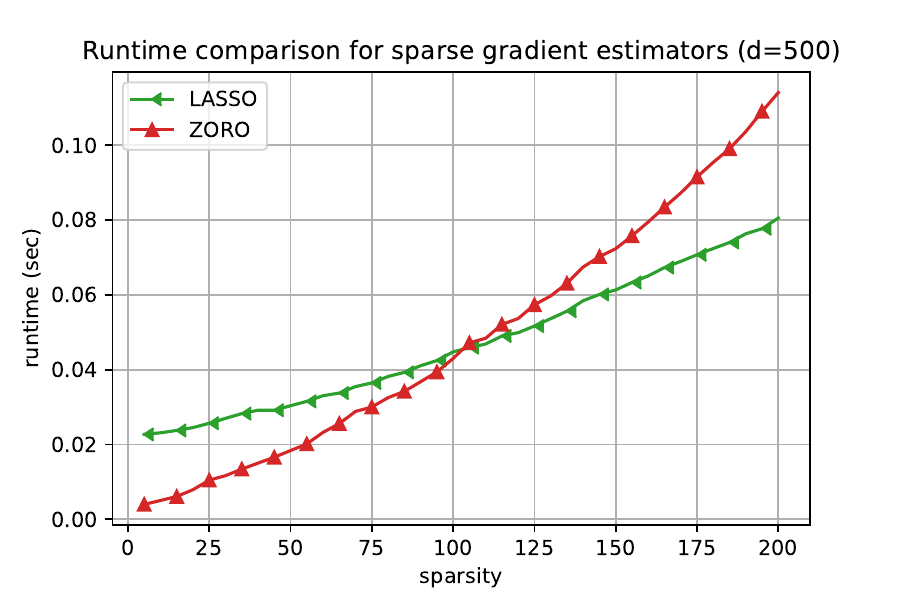} \label{fig_time2}}
\hfill
\vspace{-0.05in}
\caption{Runtime comparisons for the sparse gradient estimators in synthetic examples. } \label{fig_time_all}
\end{figure}

In this section, we investigate the computational efficiency of the sparse gradient estimators in LASSO and ZORO. The experiments were executed on a Windows 10 laptop with Intel i7-8750H CPU (6 cores at 2.2GHz) and 32GB of RAM.

We consider the quadratic function $f(x) = x^{\top} A x/2$, where $A\in \mathbb{R}^{d\times d}$ is a diagonal matrix with $s$ non-zero randomly generated positive elements. As shown in Figure~\ref{fig_quad}, LASSO has trouble estimating the gradients precisely when close to the optimal points, so we compared the speed of the sparse gradient estimators by averaging the results at $100$ randomly selected points. We use the same function queries for both gradient estimators. We emphasize that ZORO does not gain any additional advantage from a lower number of samples or better convergence in these speed experiments. The runtime per gradient estimation was evaluated with varying problem dimension (see Figure~\ref{fig_time1}) and sparsity level (see Figure~\ref{fig_time2}). 
We find the gradient estimator in ZORO is faster when the problem dimension is large and sparsity level is small while the gradient estimator in LASSO is faster when $d$ is small and $s$ is large. Since we are more interested in the high dimensional problems with sparse gradients, ZORO appears to have noticeable speed advantage in our problem setting.

\subsection{Asset risk management} \label{subsec:asset risk}

\begin{figure}[t]
\centering
\includegraphics[scale = 0.45]{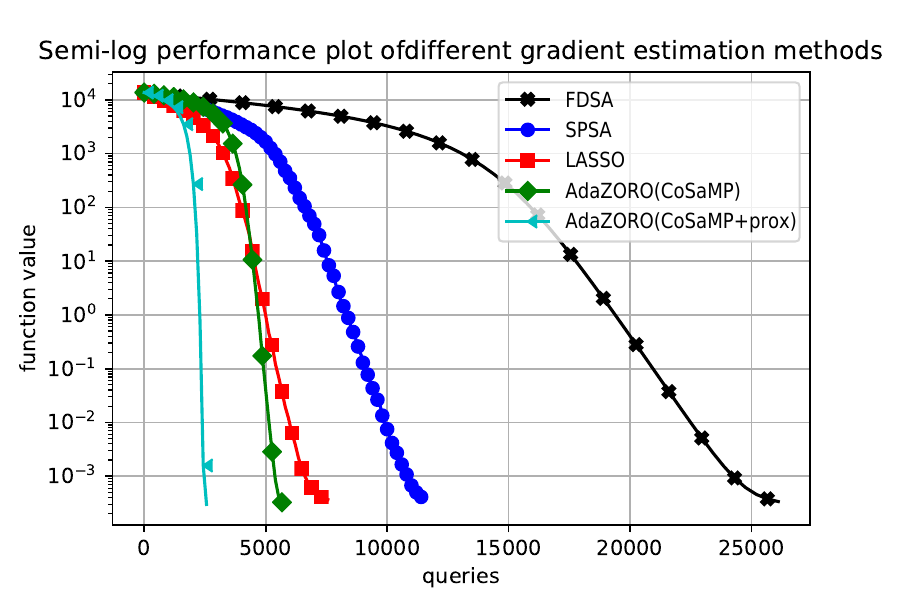} 
\vspace{-0.05in}
\caption{Function values \textit{v.s.} queries for gradient estimation methods in asset risk management.} \label{fig_asset}
\end{figure}

Consider a portfolio consisting of $d$ different assets, where $x_i$ and $m_i$ denote the fraction of the portfolio invested in, and the expected return of, asset $i$ respectively. 
$C$ denotes the covariance matrix of asset returns. The portfolio risk, which we aim to minimize, is $\frac{x^{\top}C x}{2(\sum_{i=1}^d x_i)^2}$. We used the correlation, mean, and standard deviation of 225 assets from the dataset of~\citep{CHANG20001271}. Our goal is to minimize the risk function subject to the expected return constraint 
$\frac{\sum_{i=1}^d m_i x_i}{\sum_{i=1}^d x_i}>r.$
We penalize the risk to formulate the problem as:
\begin{align*}
\minimize_{x\in\mathbb{R}^d} \frac{x^{\top}C x}{2(\sum_{i=1}^d x_i)^2} + \lambda \left(\min\left\{\frac{\sum_{i=1}^d m_i x_i}{\sum_{i=1}^d x_i}-r, 0\right\}\right)^2.
\end{align*}
Optionally, a non-negative constraint ($x_i\geq 0$ for all $i$) can be added to this problem, and it can be imposed by a proximal operator. 
In this experiment, instead of solving this quadratic program directly, consider a problem proposer who wishes to keep the formulation and data private and only offers noisy zeroth-order oracle access. 

As the appropriate $s$ is not clear {\em a priori}, we use the adaptive sampling strategy ({\em i.e.} AdaZORO) for this problem. The results are shown in Figure~\ref{fig_asset}. When $x_i$ was unconstrained, the query efficiency of AdaZORO(CoSaMP) was twice as good as SPSA and five times as good as FDSA; moreover, it saves around $20\%$ queries as compared to LASSO. When imposing the non-negativity constraint via a proximal operator, AdaZORO(CoSaMP$+$prox) further improved the query efficiency to be twice as good as LASSO. 
All solutions in this test matched the optimal value found via quadratic programming (using full knowledge of $C$), which is approximately $2\times 10^{-4}$.

\subsection{Sparse adversarial attack on ImageNet} \label{subsec:sparse attack}

\begin{figure}[t]
\centering
\subfloat[``corn'' $\rightarrow$ ``ear, spike, capitulum'']{\includegraphics[width=.242\linewidth]{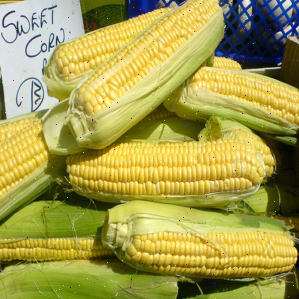}} \hfill
\subfloat[``plastic bag'' $\rightarrow$ ``shower cap'']{\includegraphics[width=.242\linewidth]{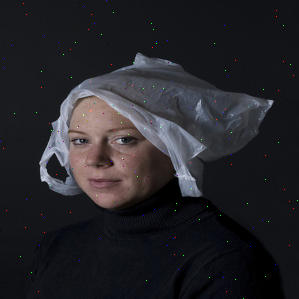}}  \hfill
\subfloat[``water ouzel, dipper'' $\rightarrow$  ``otter'']{\includegraphics[width=.242\linewidth]{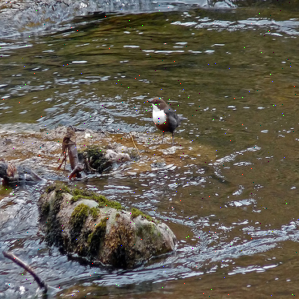}} \hfill
\subfloat[``thimble'' $\rightarrow$ ``measuring cup'']{\includegraphics[width=.242\linewidth]{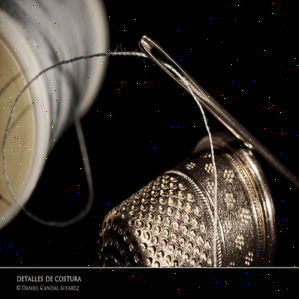}}
\vspace{-0.05in}
\caption{Examples of adversarial images, true labels and mis-classified labels.} \label{fig:sparse_attack_example}
\end{figure}

We tested generating black-box adversarial examples using ZORO. We used Inception-V3 model \citep{szegedy2016rethinking} on ImageNet \citep{deng2009imagenet} and focused on per-image adversarial attacks. The authors in \citep{chen2019zo} considered a similar problem by optimizing the attack loss and the $\ell_2$ norm of image distortion. In contrast, we aimed to find distortions $\delta$ for single images $x$ such that the attack loss $f(x+\delta)$ and the $\ell_0$ norm of distortion are minimized:
$\minimize_{\delta}  f(x+\delta) + \lambda\|\delta\|_0.$ Similar to \citep{chen2019zo}, we use ZORO to attack 100 random images from ImageNet that are correctly classified by Inception-V3. We compared ZORO with ZO-AdaMM, ZO-SGD, and ZO-SCD \citep{chen2019zo}. ZO-SCD is essentially a variation of FDSA, and ZO-SGD is a mini-batched version of SPSA. All of these methods led to extremely large $\ell_0$ distortions, except ZO-SCD, which had the worse $\ell_2$ distortion. Hence, successful attacks required distorting nearly all pixels. We used the same setup in \citep{chen2019zo}: $10$ queries at each iteration and check if the attack succeeds before $1000$ iterations. As the problem dimension is very large ($d = 65,536$) we use a block coordinate descent version of ZORO. In each iteration we randomly selected a subspace of 2000 dimensions (pixels) and generated random perturbations only in this subspace. We took $s=10$ and $m=50$, and performed $200$ iterations. Although prior sparse adversarial attacks exist ({\em e.g.} SparseFool~\citep{modas2019sparsefool}), we appear to be the first to connect adversarial attacks to sparse zeroth-order optimization. Table~\ref{sparse-attack} presents the experimental results. ZORO had the highest attack success rate while having the lowest average $\ell_0$ distortion. Surprisingly, the average $\ell_2$ distortion of ZORO was also the best. The average query complexity of ZORO is slightly worse than the other methods tested, as it uses more queries at each iteration. Some pictures of successful sparse attacks by ZORO are presented in Figure~\ref{fig:sparse_attack_example}.

\begin{table}[t]
\caption{Attack success rate (ASR), average final $\ell_0$ distortion (as a percentage of the total number of pixels), average final $\ell_2$ distortion, and average iterations of first successful attack for different zeroth-order attack methods. The image pixels are normalized to $[-0.5, 0.5]$ to calculate $\ell_2$ distortion.} \label{sparse-attack}
\begin{center}
\begin{tabular}{lcccc}
\toprule
\textsc{Methods} & \textsc{ASR} & \textsc{$\ell_0$ dist} & \textsc{$\ell_2$ dist} & \textsc{Iter} \\
\midrule
ZO-SCD    & $78 \%$ &  $0.89 \%$ & $57.5$  & $240$\\
ZO-SGD  & $78 \%$ &  $100 \%$ &  $37.9$ & $159$\\
ZO-AdaMM  & $81 \%$&  $100 \%$ &  $28.2$ & $172$\\
ZORO      & \textbf{90$\%$} & \textbf{0.73$\%$} &  \textbf{21.1} &  \textbf{59}     \\
\bottomrule
\end{tabular}
\end{center}
\end{table}

\begin{table}[t]
\caption{Recovery success rate (RSR), original image distortion rate, and total prediction accuracy reduction for different median filter sizes.} \label{median-filter}
\begin{center}
\begin{tabular}{cccc}
\toprule
\textsc{Median filter} & \textsc{RSR} & \textsc{Dist rate} & \textsc{TOT reduction}  \\
\midrule
size = 2  & $86 \%$ &  $8 \%$ &  $ 21 \%$\\
size = 3  & \textbf{92$\%$} &  \textbf{7$\%$} &  \textbf{14$\%$}\\
size = 4  & $76 \%$ &  $14 \%$ &  $34 \%$ \\
size = 5  & $69 \%$ &  $29 \%$ &  $53 \%$ \\
\bottomrule
\end{tabular}
\end{center}
\end{table}

Out of curiosity, we tested mitigating our sparse attacks by applying a median filter, a common method to remove speckle noise. 
We used Inception-V3 on the attacked-then-filtered imaged to check whether true labels are obtained, {\em i.e.} whether or not the attack has been mitigated. We also applied the same filter to the original (un-attacked) images to check if they reduced label accuracy. 
Let $\mathcal{A}$ denote the set of adversarial images that are successfully attacked by ZORO, $\mathcal{I}_1$ denote the set of image IDs that are not recovered, and $\mathcal{I}_2$ denote the set of image IDs that are mis-classified. The recovery success rate (RSR), $1- |\mathcal{I}_1|/|\mathcal{A}|$, is the ratio of images in $\mathcal{A}$ been identified to the true label after filtering. 
The distortion rate, $|\mathcal{I}_2|/|\mathcal{A}|$, is the ratio of images been assigned an incorrect label. Note that there are some overlapping IDs in $\mathcal{I}_1$ and $\mathcal{I}_2$. The total accuracy reduction, $|\mathcal{I}_1 \cup \mathcal{I}_2|/|\mathcal{A}|$, summarizes these two experiments. 
The test results are presented in Table~\ref{median-filter}. 
While mitigating many attacks, the median filter also distorted the original images, causing lower classification accuracies.

\bibliographystyle{siamplain}
\bibliography{ZerothOrderBib}

\appendix


\section{Functions with sparse gradients}
\label{sec:SparseGradients}
Let $Y\sim N(0,I_{d})$ be a zero-mean Gaussian random vector with covariance matrix the identity. Define $f_{\nu}(x) = \mathbb{E}[f(x+\nu Y)]$, the {\em Gaussian smoothing} of $f(x)$. In \citep{Balasubramanian2018} it is claimed, below Assumption~4 on pg.~7,  that if $\|\nabla f(x)\|_{0} \leq s$ then $\|\nabla f_{\nu}(x)\|_{0} \leq s$. This is false in general, as Theorem~\ref{thm:OtherPaperIsWrong} shows. Because this is not true, the following key line (see pg.~4 of the supplementary) in the proof of Lemma~C.2 of \citep{Balasubramanian2018}:
\begin{equation*}
\|\nabla f_{\nu}(x) - \nabla f(x)\|_{2} \leq \sqrt{s}\|\nabla f_{\nu}(x) - \nabla f(x)\|_{\infty} 
\end{equation*}
is not correct, and thus Lemma~C.2 is false. Because Lemma~C.2 is crucial for the proofs of Theorems~3.1 and 3.2 in \citep{Balasubramanian2018}, these Theorems are also false. As far as we can tell, the only way to fix these theorems is to replace the assumption: $\|\nabla f(x)\|_{0} \leq s$ for all $x\in\mathbb{R}^{d}$ with the more restrictive fixed support assumption: $\supp(\nabla f(x)) = S$ for all $x\in\mathbb{R}^{d}$ for some fixed $S\subset\{1,\ldots, d\}$.

\begin{theorem}
\label{thm:OtherPaperIsWrong}
Suppose $f$ is continuously differentiable and $\|\nabla f(x)\|_{0} \leq s$ for all $x\in\mathbb{R}^{d}$. If there exist $x_1,x_2\in\mathbb{R}^{d}$ with $S_1:= \supp(\nabla f(x_1)) \neq \supp(\nabla f(x_2)) =: S_2$, then $\|\nabla f_{\nu}(x)\|_{0} > s$ for all $x\in\mathbb{R}^{d}$. 
\end{theorem}

Before proving this theorem, it is useful to introduce some notation and provide some preliminary lemmas. Let $g_{\nu}(x)$ denotes the Gaussian kernel:
\begin{equation}
g_{\nu}(x) := \frac{1}{(2\pi)^{d/2}\nu^{d}}\exp\left( -\frac{\|z\|_{2}^{2}}{2\nu^{2d}} \right).
\label{eq:Gaussian_Kernel}
\end{equation}
Observe that $f_{\nu}(x)$ is the convolution of $f(x)$ and $g_{\nu}(x)$:
\begin{align*}
   f_{\nu}(x) &= \mathbb{E}[f(x+\nu Y)] \\
   & = \frac{1}{(2\pi)^{d/2}} \int_{\mathbb{R}^{d}}f(x+\nu y)\exp(-\|y\|_2^2/2)dy \quad\quad \text{ \em change variables: } z = -\nu y\\
   & = \frac{1}{(2\pi)^{d/2}} \frac{1}{\nu^{d}}\int_{\mathbb{R}^{d}} f(x - z)\exp\left( -\frac{\|z\|_{2}^{2}}{2\nu^{2d}} \right)dz = f(x)*g_{\nu}(x) .
\end{align*}

\begin{lemma}
\label{lemma:ZerosAnalytic}
Suppose $f$ is continuously differentiable and let $\nabla_{i}f_{\nu}(x) := \frac{\partial}{\partial x_i}f_{\nu}(x)$. If there exists an open set $U\subset\mathbb{R}^{d}$ such that $\nabla_{i}f_{\nu}(x) = 0$ for all $x\in U$ then $\nabla_{i}f_{\nu}(x) = 0$ for all $x\in\mathbb{R}^{d}$.
\end{lemma}

\begin{proof}
First, observe that if $f(x)$ is continuously differentiable then:
\begin{equation}
\nabla_{i}f_{\nu}(x) = \frac{\partial}{\partial x_i}\left(f*g_{\nu}\right) \stackrel{(a)}{=} \left[\frac{\partial}{\partial x_i} f\right]*g_{\nu} = \left[\nabla_{i}f\right] *g_{\nu},
\label{eq:Conv_Diff}
\end{equation}
where $(a)$ is a well-known property of convolutions. Because $g_{\nu}(x)$ is an analytic function, $\nabla_{i}f_{\nu}(x)$ is also analytic as convolution preserves analyticity. But then it follows from basic properties of analytic functions that if $\nabla_{i}f_{\nu}(x)$ is zero on an open set it is zero everywhere.
\end{proof}

We now prove the theorem by using some fundamental results in Fourier analysis:
\begin{proof}
Let $\mathcal{F}$ denote the Fourier transform. Suppose that $f,g: \mathbb{R}^{d}\to\mathbb{R}$ are continuously differentiable. We shall use the following well-known facts about  $\mathcal{F}$:
\begin{enumerate}
    \item $\mathcal{F}\left(f*g\right) = \mathcal{F}(f)\mathcal{F}(g)$. 
    
    \item If $\mathcal{F}(g) = 0$ then $g = 0$. 
    
    \item Let $g_{\nu}$ be the Gaussian kernel \eqref{eq:Gaussian_Kernel}. Then $\displaystyle \mathcal{F}(g_{\nu}) = \frac{(2\pi)}{\nu^{d}}g_{1/\nu}(x)$.
\end{enumerate}
Pick any $x\in\mathbb{R}^{d}$ and let $S_{x} := \supp\left(\nabla f_{\nu}(x)\right)$ We claim that $S_1\cup S_2 \subset S_{x}$. First, observe $f_{\nu}$ is continuously differentiable (in fact, analytic) because $g_{\nu}$ is analytic. Thus, there exists an open set $U$ containing $x$ upon which the support of $\nabla f_{\nu}$ is constant: $\supp(\nabla f_{\nu}(y)) = S_{x}$ for all $y \in U$. Pick any $i\in S_1$. If $i\notin S_{x}$ then $\nabla_{i}f_{\nu}(y) = 0$ for all $y\in U$. Then, by Lemma~\ref{lemma:ZerosAnalytic}, $\nabla_{i}f_{\nu}(y) = 0$ for all $y\in \mathbb{R}^{d}$. Now apply the Fourier transform:
\begin{align*}
    0 = \mathcal{F}(0) = \mathcal{F}\left(\nabla_{i}f_{\nu}\right) \stackrel{(a)}{=} \mathcal{F}\left(\left[\nabla_{i}f\right] *g_{\nu}\right) \stackrel{(b)}{=} \mathcal{F}\left(\nabla_{i}f\right)\mathcal{F}\left(g_{\nu}\right) \stackrel{(c)}{=} \mathcal{F}\left(\nabla_{i}f\right) \left(\frac{(2\pi)}{\nu^{d}}g_{1/\nu}\right),
\end{align*}
where (a) follows from \eqref{eq:Conv_Diff}, (b) follows from Fact~1, and (c) follows from Fact~3. In other words:
\begin{equation*}
\mathcal{F}\left(\nabla_{i}f\right)(y)g_{1/\nu}(y) = 0 \quad \text{ for all } y \in \mathbb{R}^{d}.
\end{equation*}
But, $g_{1/\nu}(y)\neq 0$ for all $y\in \mathbb{R}^{d}$, hence $\mathcal{F}\left(\nabla_{i}f\right)(y) = 0$ for all $y\in\mathbb{R}^{d}$. That is, $\mathcal{F}\left(\nabla_{i}f\right)$ is the zero function. Appealing to Fact 2 above, $\nabla_{i}f$ is also the zero function. That is, $\nabla_{i}f(y) = 0$ for all $y\in\mathbb{R}^{d}$. This is a contradiction as $i\in S_1$ so by definition $\nabla_{i}f(x_1) \neq 0$. Hence we must have $i\in S_x$, and thus $S_1\subset S_x$. The same argument implies that $S_2\subset S_x$. Because $|S_1| = |S_2| = s$ but $S_1\neq S_2 $ we have that $|S_x| > s$, thus proving the theorem.
\end{proof}

\begin{theorem}
\label{thm:Sparse_Grad_Strong_Convex}
Suppose $f$ is continuously differentiable and $\|\nabla f(x)\|_{0} \leq s$ for all $x\in\mathbb{R}^{d}$. Then $f$ cannot be strongly convex.
\end{theorem}

\begin{proof}
Recall if $f$ is strongly convex then there exists a $\mu > 0$ such that for all $x,y\in\mathbb{R}^{d}$:
\begin{equation}
f(y) - f(x) \geq \nabla f(x)^{\top}(y-x) + \frac{\mu}{2}\|y-x\|_2^2.
\label{eq:Res_Convex_Property}
\end{equation}
Pick any $x\in\mathbb{R}^{d}$ and let $S := \supp(\nabla f(x))$. As in the proof of Theorem~\ref{thm:OtherPaperIsWrong}, because $f$ is continuously differentiable there exists an open set $U$ such that $\supp(\nabla f(y)) = S$ for all $y\in U$. Shrinking $U$ further if necessary we may assume $U$ is convex. Pick any $i \notin S$ and let $\bfe_i$ denote the $i$-th canonical basis vector; then clearly $\nabla f(z)^{\top}\bfe_i = 0$ for all $z\in U$. Choose $y = x + \varepsilon\bfe_i$, where $\varepsilon$ is small enough such that $y\in U$. By the fundamental theorem of calculus:
\begin{equation*}
f(y) - f(x) = \int_{0}^{\varepsilon}t\nabla f(x+t\bfe_i)^{\top}\bfe_i dt \stackrel{(a)}{=} \int_{0}^{\varepsilon}t(0)dt = 0,
\end{equation*}
where (a) holds because $U$ is convex and thus $x+t\bfe_i\in U$ for all $0\leq t \leq \varepsilon$. Moreover, $\nabla f(x)^{\top}(y-x) = 0$. Returning to \eqref{eq:Res_Convex_Property}, if $f$ were strongly convex we would have $0 \geq 0 + \frac{\mu}{2}\|y-x\|_2^2$ for some $\mu > 0$, a clear contradiction.
\end{proof}

\end{document}